\documentclass[10pt]{article}%
\usepackage{amsfonts}
\usepackage{amssymb,amsmath,amsthm, amsfonts}

\providecommand{\U}[1]{\protect \rule{.1in}{.1in}}
\newtheorem{theorem}{Theorem}[section]

\newtheorem{corollary}{Corollary}[section]

\newtheorem{lemma}{Lemma}[section]

\newtheorem{proposition}{Proposition}[section]
\newtheorem{remark}{Remark}[section]

\def\sup{\mathop{\rm sup}}
\makeatletter
   
   \@addtoreset{equation}{section}
\makeatother

\begin{document}

\title{$L^p$ estimates for fully coupled FBSDEs with jumps {\footnote {The work has been supported by the NSF of
P.R.China (No. 11071144, 11171187, 11222110), Shandong Province
(No. BS2011SF010, JQ201202), SRF for ROCS (SEM), supported by Program for New Century Excellent Talents in University (NCET, 2012), 111 Project (No.
B12023).}}}
\author{Juan Li\\{\small School of Mathematics and Statistics, Shandong University, Weihai,
Weihai 264209, P. R. China.}\\{\small \textit{E-mail: juanli@sdu.edu.cn}}\\Qingmeng Wei\\{\small School of Mathematics, Shandong University, Jinan 250100, P. R.
China.}\\{\small \textit{E-mail: qingmengwei@gmail.com}}}

\date{January 29, 2013}
\maketitle

\bigskip

\noindent \textbf{Abstract.}  In this paper  we study useful
estimates, in particular $L^p$-estimates, for fully coupled
forward-backward stochastic differential equations (FBSDEs) with
jumps. These estimates are proved at one hand for fully coupled
FBSDEs with jumps under the monotonicity assumption for arbitrary time
intervals and on the other hand for
such equations on small time intervals. Moreover, the well-posedness of this kind of equation is
studied and regularity results are obtained.

\bigskip


\section{{\protect \large {Introduction}}}

General nonlinear backward stochastic differential equations (BSDEs,
for short) driven by a Brownian motion were introduced and studied
by Pardoux, Peng in \cite{PaPe1}. Since that pioneering paper from
1990, the theory of BSDEs has been intensively studied by a lot of
researchers attracted by its various applications, namely in
stochastic control (see Peng \cite{Pe2}), finance (see El Karoui,
Peng and Quenez \cite{ELPeQu}), and the theory of partial
differential equations (PDEs, for short) (see Pardoux, Peng
\cite{PaPe2}, Peng \cite{Pe3}, etc).

The study of BSDEs has led also to generalizations, among them BSDEs
driven by both a Brownian motion and an independent Poisson random
measure (first studied by Tang and Li \cite{TL}) but also fully
coupled forward-backward stochastic differential equations (FBSDEs)
governed by a Brownian motion and such FBSDEs governed by both a
Brownian motion and Poisson random measure.

As concerns the fully coupled FBSDEs driven by a Brownina motion,
they were intensively studied under different assumptions by
different authors. While Ma and Yong \cite{MY} developed under the
assumption of strict ellipticity of the diffusion coefficient of the
forward equation the so-called 4-step scheme for FBSDE, Hu and Peng
\cite{HP}, Peng and Wu \cite{PW} studied FBSDEs under the so-called monotonicity
assumption, while Pardoux and Tang \cite{PaT} used a different condition. All
these three conditions are of different type and not really
comparable. In recent works Ma, Wu, Zhang and Zhang \cite{Ma-Wu-Zhang-Zhang} have studied
fully coupled FBSDE which involve these three types of conditions.

Fully coupled FBSDEs driven by both a Brownian motion and a Poisson
random measure were studied by Wu \cite{W1999}, \cite{W2003} under
the monotonicity condition. For this he extended the arguments of
\cite{HP}, \cite{PW} to the case with jumps.  While in \cite{W1999} he obtained
the existence and the uniqueness for such fully coupled FBSDEs with
jumps, in Wu \cite{W2003} he proved the existence and the uniqueness
of the solution as well as a comparison theorem for fully coupled
FBSDEs with jumps over a stochastic interval.

The main objective of our paper is to study useful estimates, in
particular $L^p$ estimates for fully coupled FBSDEs with jumps which are not the same as $L^p$ estimates for fully coupled FBSDEs driven only by a
Brownian motion, refer to Proposition 3.2, Remark 3.4, and Theorem 3.4. These estimates, particularly
challenging for the case of fully coupled FBSDEs with jumps, have
been already well studied for fully coupled FBSDEs driven only by a
Brownian motion. We refer the reader, in particular, to the paper
\cite{D} by Delarue. His results and estimates for fully coupled
FBSDEs driven only by a Brownian motion over a sufficiently small time interval were extended by Li
and Wei \cite{LW} to controlled fully coupled FBSDEs in the frame of
their study of an optimal stochastic control problem with coupling
between the controlled forward and the controlled backward equation, while, in particular, the diffusion coefficient of the forward equation $\sigma$\ depends on $z$.
In the frame of their studies they proved some new $L^p$-estimates
for fully coupled FBSDEs on small time interval which were crucially
used for the link between the stochastic control problem and the
associated system of PDEs formed by a quasi-linear
Hamilton-Jacobi-Bellman (HJB, for short) equation and an algebraic equation.

Inspired by the control problems studied by \cite{BLH}, \cite{LP} and \cite{LW}, Li, Wei \cite{LiWei-SDG} have investigated
recently stochastic differential games defined through fully coupled
FBSDEs with jumps. These studies have required specific types of non-trivial $L^p$-estimates for fully coupled FBSDE with jumps, which
have also their own interest. They extend former results for coupled
FBSDEs without jumps and are based on rather technical proofs.

In this paper, we first study $L^2$-estimates (Proposition
\ref{pro6.1}) and $L^p$-estimates (Proposition
\ref{pro6.2}) for fully coupled FBSDEs with
jumps under the monotonicity condition. In our proofs we use a new
method, in particular in the proof of Proposition \ref{pro6.2}; the
estimates (\ref{ee6.1111}) and (\ref{ee6.11111}) concerning the jump
martingale part turn out to be crucial for other estimates in this
work.

In the second part of our paper, assuming the Lipschitz coefficients
with respect to $z$ and $k$ of the diffusion coefficient and the
coefficient in the jump integral to be sufficiently small, we first
prove the existence and uniqueness (Theorem \ref{pro6.3}) of the
solution of fully coupled FBSDEs with jumps on a small time interval
and also a generalized Comparison Theorem (Theorem \ref{th6.2}).
Then we derive the $L^p$-estimates (Theorem \ref{pro6.4}) for fully
coupled FBSDEs with jumps on the small time interval. This second
part provides estimates which turn out to be crucial in the study of
stochastic differential games and for the study of the existence of
the viscosity solution for the associated second order
integral-partial differential equation of Isaacs' type over an
arbitrary time interval, combined with an algebraic equation; see
\cite{LiWei-SDG}. Of course, the results of our paper can be also
applied to the study of other problems, as for instance, the optimal
control problems and the stochastic maximum principle of fully
coupled FBSDEs with jumps.

This paper is organized as follows: In Section 2 we recall some
preliminaries for  fully coupled FBSDEs with jumps, which will be
used later. In Section 3, on one hand, we prove some basic estimates
for fully coupled FBSDEs with jumps under monotonicity condition, on
the other hand,  assuming the Lipschitz coefficients of $\sigma, \
h$ with respect to $z,\ k$ to be sufficiently small, we establish
the well-posedness result and a generalized Comparison Theorem for
fully coupled FBSDEs with jumps on a small time interval. The
associated $L^p$-estimates ($p\geq2$) are then
derived.

\section{ {\protect \large Preliminaries}}

 Let $(\Omega, {\mathcal{F}}, \{\mathcal {F}_t\}_{t\geq 0}, P)$ be a complete probability space, where $\mathbb{F}=\{\mathcal {F}_t\}_{t\geq 0}$ is a natural filtration generated by the following two mutually independent processes, and completed by all $P$-null sets:

(i) a $d$-dimensional standard Brownian motion $\{B_t\}_{t\geq 0}$;

(ii) a Poisson random measure $\mu$ on $\mathbb{R}^+\times E$, where $E=\mathbb{R}^{l}\backslash \{0\}$ is equipped
with its Borel $\sigma$-field $\mathcal{B}(E)$, with the compensator $\hat{\mu}(dt, de)=dt\lambda(de)$ such that $\{\tilde{\mu}((0,t]\times A)=(\mu-\hat{\mu})((0,t]\times A)\}_{t\geq 0}$ being a martingale for all
$A\in \mathcal {B}(E)$ satisfying $\lambda(A) < \infty$.
  Here $\lambda$ is assumed to be a
$\sigma$-finite L\'{e}vy measure on $(E,\mathcal{B}(E))$ with the property
that $\int_{E}(1\wedge|e|^{2})\lambda(de)<\infty.$

For any $n\geq1,\ |z|$ denotes the Euclidean norm of $z\in \mathbb{R}^{n}.$  Fix $T>0$, and $[0, T]$ is called the time duration. Now we give some spaces of processes which will be used later:
\begin{itemize}
\item $\mathcal{M}^{2}(t,T;\mathbb{R}^{d}):=\Big \{ \varphi \mid \varphi
:\Omega \times[t,T]\rightarrow \mathbb{R}^{d} \mbox{ is an }\mathbb{F}
\mbox{-predictable process}:\ \parallel \varphi \parallel^{2}=E[\int^{T}_{t} |\varphi_{s}%
|^{2}ds]<+\infty \Big \}; $

\item ${\mathcal{S}}^{2}(t,T;\mathbb{R}):=\Big \{ \psi \mid \psi:\Omega
\times[t,T]\rightarrow \mathbb{R} \mbox{ is an } \mathbb{F}
\mbox{-adapted c\`{a}dl\`{a}g process}:\ E[\mathop{\rm
sup}\limits_{t\leq s\leq T}| \psi_{s} |^{2}]< +\infty \Big \}; $

\item $\mathcal{K}_{\lambda}^{2}(t,T;\mathbb{R}^{n}):=\Big \{K\mid
K:\Omega \times[t,T]\times E\rightarrow \mathbb{R}^{n} \mbox{ is
}\mathcal{P}\otimes \mathcal{B}(E)-\mbox{measurable}:\newline \mbox{ }\hskip3cm
\parallel K\parallel^{2}=E[\int_{t}^{T}\int
_{E}|K_{s}(e)|^{2}\lambda(de)ds]<+\infty \Big
\},$
\end{itemize}
where $t\in[0,T].$ Here  $\mathcal {P}$ denotes the $\sigma$-field of
$\mathbb{F}$-predictable subsets of $\Omega\times [0,T].$

\subsection{{\protect \large Fully coupled FBSDEs with jumps}}

 Now we consider the following fully coupled FBSDE with
jumps associated with $(b, \sigma, h, f, \zeta, \Phi)$\ on the time interval $[t, T]$ ($t\in [0, T]$):
\begin{equation}
\left \{
\begin{array}
[c]{llll}%
dX_{s} & = & b(s,X_{s},Y_{s},Z_{s},K_{s})ds+\sigma(s,X_{s},Y_{s},Z_{s}%
,K_{s})dB_{s} +\int_{E}h(s,X_{s-},Y_{s-},Z_{s},K_{s}(e),e)\tilde{\mu
}(dsde), & \\
dY_{s} & = & -f(s,X_{s},Y_{s},Z_{s},\int_{E}K_{s}(e)l(e)\lambda(de))ds+Z_{s}dB_{s}+\int_{E}K_{s}(e)\tilde{\mu
}(dsde),\  \  \ s\in \lbrack t,T], & \\
X_{t} & = & \zeta, & \\
Y_{T} & = & \Phi(X_{T}), &
\end{array}
\right.  \label{equ2.2}%
\end{equation}
where the solution $(X, Y, Z, K)$\ takes its values in $ \mathbb{R}^{n}\times \mathbb{R}^{m}\times
\mathbb{R}^{m\times d}\times \mathbb{R}^{m}$, and the coefficients
 $$\begin{array}{llll}
&&b:\Omega \times \lbrack0,T]\times \mathbb{R}^{n}\times
\mathbb{R}^{m}\times \mathbb{R}^{m\times d}\times
L^{2}(E,\mathcal{B}(E),\lambda;\mathbb{R}^m)\longrightarrow \mathbb{R}^{n}, \\
&& \sigma
:\Omega \times \lbrack0,T]\times \mathbb{R}^{n}\times \mathbb{R}^{m}%
\times \mathbb{R}^{m\times d}\times L^{2}(E,\mathcal{B}(E),\lambda;\mathbb{R}^m)\longrightarrow \mathbb{R}%
^{n\times d},\\
&& h:\Omega \times \lbrack0,T]\times \mathbb{R}^{n}%
\times \mathbb{R}^{m}\times \mathbb{R}^{m\times d}\times
\mathbb{R}^{m}\times E\longrightarrow \mathbb{R}^{n},\\
&&f:\Omega
\times \lbrack0,T]\times \mathbb{R}^{n}\times \mathbb{R}^{m}\times
\mathbb{R}^{m\times d}\times \mathbb{R}^{m}\longrightarrow
\mathbb{R}^{m}, \end{array}$$
 $l:E\longrightarrow \mathbb{R}$ and $\Phi:\Omega \times \mathbb{R}^{n}%
\longrightarrow \mathbb{R}^{m}$ satisfy
\begin{description}
\item[$( \mathbf{H2.1})$] (i) $b,\  \sigma,\ f$ are uniformly Lipschitz with respect to $(x,y,z,k),$ and there exists $\rho:E\rightarrow
\mathbb{R}^{+}$ with $\int_{E}\rho^{2}(e)\lambda(de)<+\infty$
such that, for any $t\in[0,T],\ x,\bar{x}\in \mathbb{R}^{n},\ y,\bar{y}%
\in \mathbb{R}^{m},\ z,\bar{z}\in \mathbb{R}^{m\times d},\ k,\bar{k}%
\in \mathbb{R}^{m}$ and $e\in E$,
\[
|h(t,x,y,z,k,e)-h(t,\bar{x},\bar{y},\bar{z},\bar{k},e)|\leq \rho(e)(|x-\bar
{x}|+|y-\bar{y}|+|z-\bar{z}|)+C|k-\bar{k}|;
\]
(ii) $k\rightarrow f(t,x,y,z,k)$ is non-decreasing, for all $(t,x,y,z)\in
[0,T]\times \mathbb{R}^{n}\times \mathbb{R}^{m}\times \mathbb{R}^{m\times d}; $

(iii) there exists a constant $C>0$ such that
\[
0\leq l(e)\leq C(1\wedge|e|),\ x\in \mathbb{R}^{n},\ e\in E;
\]
(iv) $\Phi(x)$ is uniformly Lipschitz with respect to $x\in \mathbb{R}^{n};$

(v) for every $(x,y,z,k)\in \mathbb{R}^{n}\times
\mathbb{R}^{m}\times \mathbb{R}^{m\times d}\times \mathbb{R}^{m}, \
\Phi(x)\in L^{2}(\Omega ,\mathcal{F}_{T},P;\mathbb{R}^m)$, $b,\
\sigma,\ h,\ f$ are $\mathbb{F}$-progressively measurable and
\[%
\begin{array}
[c]{ll}%
E\int_{0}^{T}|b(s,0,0,0,0)|^{2}ds+E\int_{0}%
^{T}|f(s,0,0,0,0)|^{2}ds+E\int_{0}^{T}|\sigma(s,0,0,0,0)|^{2}%
ds & \\
+E\int_{0}^{T}\int_{E}|h(s,0,0,0,0,e)|^{2}%
\lambda(de)ds<\infty. &
\end{array}
\]
\end{description}

Let
\[
g(s,x,y,z,k):=f(s,x,y,z,\int_{E}k(e)l(e)\lambda(de)),
\]
$(s,x,y,z,k)\in[0,T]\times \mathbb{R}^{n}\times \mathbb{R}^{m}\times
\mathbb{R}^{m\times d}\times L^{2}(E,\mathcal{B}(E),\lambda;\mathbb{R}).$

In this paper we use the usual inner product and the Euclidean norm in
$\mathbb{R}^{n},\  \mathbb{R}^{m}$ and $\mathbb{R}^{m\times d},$ respectively.
Given an $m \times n$ full-rank matrix $G$, we define:
\[
\pi= \  \left(
\begin{array}
[c]{c}%
x\\
y\\
z
\end{array}
\right)  \ , \  \  \  \  \  \  \  \  \  \ A(t,\pi,k)= \  \left(
\begin{array}
[c]{c}%
-G^{T}g\\
Gb\\
G\sigma
\end{array}
\right)  (t,\pi,k),
\]
where $G^{T}$ is the transposed matrix of $G$.

We assume the following monotonicity conditions:
\begin{description}
\item[$( \mathbf{H2.2})$] {\rm (i)}
 $\begin{array}[c]{llll}&&\langle
A(t,\pi,k)-A(t,\bar{\pi},\bar{k}),\pi-\bar{\pi} \rangle
+\int_{E}\langle G\widehat{h}(e), \widehat{k}(e)\rangle
\lambda(de) \\
&& \leq-\beta_{1}|G\widehat{x}|^{2}-\beta_{2}(  |G^{T} \widehat{y}%
|^{2}+|G^{T} \widehat{z}|^{2})-\beta_{3}\int_{E}|G^{T} \widehat{k}%
(e)|^{2}\lambda(de),
\end{array}
$

(ii) $\langle \Phi(x)-\Phi(\bar{x}),G(x-\bar{x}) \rangle \geq \mu_{1}%
|G\widehat{x}|^{2},\  \forall \pi=(x,y,z),\  \bar{\pi}=(\bar{x},\bar{y},\bar
{z}),\  \widehat{x}=x-\bar{x},\  \widehat{y}=y-\bar{y},\  \widehat{z}=z-\bar
{z},\  \widehat{k}=k-\bar{k},\  \widehat{h}(e)=h(t,\pi,k,e)-h(t,\bar{\pi}%
,\bar{k},e)$, \newline where $\beta_{1},\ \beta_{2},\ \beta_{3},\
\mu_{1}$ are nonnegative constants with $\beta_{1} + \beta_{2}>0,\
\beta_{1} + \beta_{3}>0,\ \beta_{2} + \mu_{1}>0,\ \beta_{3} +
\mu_{1}>0$. Moreover, we have $\beta_{1}>0,\ \mu_{1}>0 \
(\mbox{resp., }\beta_{2}>0,\ \beta_{3}>0)$, when $m>n$ (resp.,
$m<n$).
\end{description}

\begin{remark}
\begin{description}\item[$( \mathbf{H2.2})$-(ii)']  $( \mathbf{H2.2})$ {\rm(ii)} results in the weaker condition:
$\langle \Phi(x)-\Phi(\bar{x}),G(x-\bar{x})
\rangle \geq0$, for all $x,\ \bar{x}\in\mathbb{R}^n.$ \end{description}
\end{remark}

When $\Phi(x)=\xi\in L^{2}(\Omega,\mathcal{F}_{T},P;\mathbb{R}^{m})$, $( \mathbf{H2.2})$-(i) can be weaken as follows:

\begin{description}\item[$( \mathbf{H2.3})$] $\langle
A(t,\pi,k)-A(t,\bar{\pi},\bar{k}),\pi-\bar{\pi} \rangle
+\int_{E}\langle G\widehat{h}(e), \widehat{k}(e)\rangle
\lambda(de)  \leq-\beta_{1}|G\widehat{x}|^{2}-\beta_{2}|G^{T} \widehat{y}
|^{2},$

where $\beta_1,\ \beta_2$ are nonnegative constants with $\beta_1 +
\beta_2>0$. Moreover, we have $\beta_1>0$ (resp., $\beta_2>0$), when
$m>n$ (resp., $m<n$).
\end{description}

\begin{lemma}
\label{l3} Under the assumptions $( \mathbf{H2.1})$ and $(\mathbf{H2.2})$, for
any $\zeta \in L^{2}(\Omega,\mathcal{F}_{t},P;\mathbb{R}^{n})$, FBSDE
(\ref{equ2.2}) has a unique adapted solution $(X_{s},Y_{s},Z_{s},K_{s})_{s
\in[t,T]}\in{\mathcal{S}}^{2}(t, T; {\mathbb{R}^{n}})\times{\mathcal{S}}%
^{2}(t, T; {\mathbb{R}^{m}})\times{\mathcal{M}}^{2}(t, T; {\mathbb{R}^{m\times
d}})\times \mathcal{K}_{\lambda}^{2}(t, T; \mathbb{R}^{m}).$
\end{lemma}

\begin{lemma}
\label{le2.2} Under the assumptions $( \mathbf{H2.2})$-{\rm(ii)'} and $(\mathbf{H2.3})$, for
any $\zeta \in L^{2}(\Omega,\mathcal{F}_{t},P;\mathbb{R}^{n})$ and the terminal condition $\Phi(x)=\xi\in L^{2}(\Omega,\mathcal{F}_{T},P;\mathbb{R}^{m})$, FBSDE
(\ref{equ2.2}) has a unique adapted solution $(X_{s},Y_{s},Z_{s},K_{s})_{s
\in[t,T]}\in{\mathcal{S}}^{2}(t, T; {\mathbb{R}^{n}})\times{\mathcal{S}}%
^{2}(t, T; {\mathbb{R}^{m}})\times{\mathcal{M}}^{2}(t, T; {\mathbb{R}^{m\times
d}})\times \mathcal{K}_{\lambda}^{2}(t, T; \mathbb{R}^{m}).$
\end{lemma}

For the proof, the reader can refer to Wu \cite{W1999, W2003}.

\section{{\protect \large {Regularity results for solutions of fully coupled
FBSDEs with jumps}}}

 In this section we will study some important estimates for solutions of fully coupled FBSDEs with jumps.
\subsection{{\protect \large {Regularity results under the monotonicity
condition}}}

 First, we derive some useful estimates for the solutions
under the monotonicity condition.

Let now be given the mappings
$$
\begin{array}
[c]{ll}
& \! \! \! \! \! b: \Omega \times[0,T] \times \mathbb{R}^{n} \times \mathbb{R}
\times \mathbb{R}^{d}\times L^{2}(E,\mathcal{B}(E),\lambda;\mathbb{R}) \longrightarrow \mathbb{R}^{n},\\
& \! \! \! \! \!  \sigma: \Omega \times[0,T] \times \mathbb{R}^{n}
\times \mathbb{R} \times \mathbb{R}^{d}\times L^{2}(E,\mathcal{B}(E),\lambda;\mathbb{R}) \longrightarrow
\mathbb{R}^{d},\\
& \! \! \! \! \! h: \Omega \times[0,T] \times \mathbb{R}^{n} \times \mathbb{R}
\times \mathbb{R}^{d}\times \mathbb{R} \longrightarrow \mathbb{R}^{n},\\
& \! \! \! \! \!  g: \Omega \times[0,T] \times \mathbb{R}^{n} \times \mathbb{R}
\times \mathbb{R}^{d}\times \mathbb{R} \longrightarrow \mathbb{R},
\end{array}
$$ and  $\Phi: \Omega \times \mathbb{R} \longrightarrow
\mathbb{R}$ satisfying $( \mathbf{H2.1}), \ ( \mathbf{H2.2}),$ and also assume
\begin{description}
\item[$( \mathbf{H3.1})$] For any $t\in[0,T],$ for any $(x,y,z,k)\in
\mathbb{R}^{n}\times \mathbb{R}\times \mathbb{R}^{d}\times L^{2}(E,\mathcal{B}(E),\lambda;\mathbb{R}),\  \mbox{P-a.s.},$
$$
 |b(t,x,y,z,k)|+|\sigma(t,x,y,z,k)|+ |g(t,x,y,z,k)|+|\Phi(x)| \leq L(1+|x|+|y|+|z|+|k|), $$
and there exists a measurable function $\rho: E\rightarrow
\mathbb{R}^{+}$ with $\int _{E}\rho^{2}(e)\lambda(de)<+\infty$ such
that, for any $t\in
[0,T],\ (x,y,z,$ $k)\in \mathbb{R}^{n}\times \mathbb{R}\times \mathbb{R}^{d}%
\times \mathbb{R}$ and $e\in E$,
$$
|h(t,x,y,z,k,e)|\leq \rho(e)(1+|x|+|y|+|z|+|k|).
$$
\end{description}

We consider the following fully coupled FBSDE with jumps, parameterized by the
initial condition $(t,\zeta) \in[0,T] \times L^{2}(\Omega,\mathcal{F}%
_{t},P;\mathbb{R}^{n}): $
\begin{equation}
\label{equ6.1}\left \{
\begin{array}
[c]{llll}%
dX_{s}^{t,\zeta} & = & b(s,\Pi_{s}^{t,\zeta},K_{s}^{t,\zeta})ds + \sigma
(s,\Pi_{s}^{t,\zeta},K_{s}^{t,\zeta}) dB_{s}+\int_{E}%
h(s,\Pi_{s-}^{t,\zeta},K_{s}^{t,\zeta}(e),e) \tilde{\mu}(dsde), & \\
dY_{s}^{t,\zeta} & = & -g(s,\Pi_{s}^{t,\zeta},K_{s}^{t,\zeta})ds +
Z_{s}^{t,\zeta}dB_{s}+\int_{E}K_{s}^{t,\zeta}(e)\tilde{\mu
}(dsde), \  \  \  \  \ s\in[t,T], & \\
X_{t}^{t,\zeta} & = & \zeta, & \\
Y_{T}^{t,\zeta} & = & \Phi(X_{T}^{t,\zeta}), &
\end{array}
\right.
\end{equation}
where we have put  $\Pi_s^{t,\zeta}=(X_s^{t,\zeta},Y_s^{t,\zeta},Z_s^{t,\zeta})$,
and $\Pi_{s-}^{t,\zeta}=(X^{t,\zeta}_{s-}, Y^{t,\zeta}_{s-},
Z_s^{t,\zeta})$.

\begin{proposition}
\label{pro6.1} Under the assumptions $( \mathbf{H2.1}),\ ( \mathbf{H2.2}),\ (
\mathbf{H3.1})$, for any $0 \leq t \leq T$ and any associated initial states
$\zeta,\  \zeta^{\prime}\in L^{2}(\Omega,\mathcal{F}_{t},P;\mathbb{R}^{n}),$ we
have the following estimates, \mbox{P-a.s.}:
$$
\begin{array}
[c]{llll}%
\mathrm{(i)} & E[\mathop{\rm sup}\limits_{t\leq s\leq T}|{X}_{s}^{t,\zeta}%
-{X}_{s}^{t,\zeta^{\prime}}|^{2}+\mathop{\rm sup}\limits_{t\leq s\leq T}%
|{Y}_{s}^{t,\zeta}-{Y}_{s}^{t,\zeta^{\prime}}|^{2} + \int
_{t}^{T}|{Z}_{s}^{t,\zeta}-{Z}_{s}^{t,\zeta^{\prime}}|^{2}ds &  & \\
& +\int_{t}^{T}\int_{E}|{K}_{s}^{t,\zeta}%
(e)-{K}_{s}^{t,\zeta^{\prime}}(e)|^{2}\lambda(de)ds\mid \mathcal{F}_{t}] \leq
C|\zeta- \zeta^{\prime}|^2, &  & \\
\mathrm{(ii)} & E[\mathop{\rm sup}\limits_{t\leq s\leq T}|X_{s}^{t,\zeta}%
|^{2}+\mathop{\rm sup}\limits_{t\leq s\leq T}|Y_{s}^{t,\zeta}|^{2}%
+\int_{t}^{T}|Z_{s}^{t,\zeta}|^{2}ds+\int_{t}%
^{T} \int_{E}|{K}_{s}^{t,\zeta}(e)|^{2}\lambda(de)ds\mid
\mathcal{F}_{t}] \leq C(1 +|\zeta|^{2}). &  &
\end{array}
$$ If $\sigma,\ h$ also satisfy: \newline$( \mathbf{H3.2})$ for any
$t\in[0,T]$,
for any $(x,y,z,k)\in \mathbb{R}^{n}\times \mathbb{R}\times \mathbb{R}^{d}%
\times \mathbb{R},\  \mbox{P-a.s.},$ $|\sigma(t,x,y,z,k)|\leq
L(1+|x|+|y|),\ |h(t,x,y,z,k,e)|\leq \rho(e)(1+|x|+|y|)$, \newline then we can get

{\rm(iii)}\  \ $E[\mathop{\rm sup}\limits_{t\leq s\leq t+\delta}%
|X_{s}^{t,\zeta}-\zeta|^{2}\mid \mathcal{F}_{t}]\leq C\delta(1 +|\zeta
|^{2}),\  \mbox{P-a.s.},\ 0\leq \delta \leq T-t.$
\end{proposition}

\begin{proof}  From Lemma \ref{l3}, we know
there exist the unique solutions $(\Pi^{t,\zeta},K^{t,\zeta})\in{\cal{S}}^2(t, T;
{\mathbb{R}^n})\times{\cal{S}}^2(t, T; {\mathbb{R}})\times
{\cal{M}}^2(t, T; {\mathbb{R}^{d}})$ $\times \mathcal
{K}_\lambda^2(t, T; \mathbb{R}),$\ and
$(\Pi^{t,\zeta'},K^{t,\zeta'})\in{\cal{S}}^2(t, T;
{\mathbb{R}^n})\times{\cal{S}}^2(t, T; {\mathbb{R}})\times
{\cal{M}}^2(t, T; {\mathbb{R}^{d}})\times \mathcal
{K}_\lambda^2(t, T; \mathbb{R})$ for FBSDE (\ref{equ6.1})
associated with $\zeta$ and $\zeta'$, respectively. For convenience, we define
$$\begin{array}{llll}&&\hat{X}_s:=X_s^{t,\zeta}-X_s^{t, {\zeta'}}, \quad \hat{Y}_s:=Y_s^{t,\zeta}-Y_s^{t,{\zeta'}},\quad \hat{Z}_s:=Z_s^{t,\zeta}-Z_s^{t,{\zeta'}},\quad \hat{K}_s:=K_s^{t,\zeta}-K_s^{t,{\zeta'}},\\
&& \Delta
l(s):=l(s,\Pi_s^{t,\zeta},K_s^{t,\zeta})-l(s,\Pi_s^{t,\zeta'},K_s^{t,{\zeta'}}),\ \Delta
h(s,e):=h(s,\Pi_{s-}^{t,\zeta},K_{s}^{t,\zeta}(e),e)-h(s,\Pi_{s-}^{t,\zeta'},K_{s}^{t,{\zeta'}}(e),e),\end{array}$$
where $l=b,\  \sigma, \ g,\ A,$ respectively.\\
Applying It\^{o}'s formula to $|\hat{X}_s|^2$, we obtain from the
Gronwall inequality,
\begin{equation}\label{ee6.1} E[|\hat{X}_s|^2\mid \mathcal{F}_t] \leq
C(|\zeta-\zeta'|^2+E[\int_t^s(|\hat{Y}_r|^2+|\hat{Z}_r|^2+
\int_E|\hat{K}_r(e)|^2\lambda(de))dr|\mathcal {F}_t]),\  t\leq s\leq
T. \end{equation} Then, applying It\^{o}'s formula to $e^{\beta
s}|\hat{Y}_s|^2$, taking $\beta$ large enough, and taking into
account (\ref{ee6.1}), we get
\begin{equation}\label{ee6.3}
\begin{array} [c]{llll}&& E[|\hat{Y}_s|^2\mid \mathcal
{F}_t]+E[\int_s^T|\hat{Y}_r|^2dr+ \int_s^T|\hat{Z}_r|^2dr+\int_s^T
\int_E|\hat{K}_r(e)|^2\lambda(de)dr|\mathcal
{F}_t]\\
&\leq& C|\zeta-\zeta'|^2+CE[\int_t^T(|\hat{Y}_r|^2+|\hat{Z}_r|^2+
\int_E|\hat{K}_r(e)|^2\lambda(de))dr|\mathcal {F}_t],\  t\leq s\leq
T.\end{array}\end{equation} On the other hand, applying It\^{o}'s
formula to $\langle
G\hat{X}_r,\hat{Y}_r\rangle$, from the assumption $( \mathbf{H2.2})$ we get
\begin{equation}\label{ee6.4} \begin{array} [c]{llll} \langle
G\hat{X}_s,\hat{Y}_s\rangle
&=&E[\langle G\hat{X}_T,\hat{Y}_T\rangle \mid \mathcal {F}_s]-
E[\int_s^T(\langle \Delta
A(r),(\hat{X}_r,\hat{Y}_r,\hat{Z}_r)\rangle + \int_E\langle G\Delta
h(r,e),\hat{K}_r(e)\rangle \lambda(de))
dr|\mathcal {F}_s]\\
&\geq&E[\mu_1|G\hat{X}_T|^2\mid \mathcal
{F}_s]+E[\beta_1\int_s^T|G\hat{X}_r|^2dr|\mathcal
{F}_s]\\
&&+E[ \int_s^T\beta_2(|G^T\hat{Y}_r|^2+|G^T\hat{Z}_r|^2)\mid\mathcal {F}_s]+
E[\int_s^T\int_E\beta_3|G^T\hat{K}_r(e)|^2\lambda(de)dr\mid\mathcal {F}_s],
\end{array} \end{equation}
Therefore, $ \langle
G\hat{X}_s,\hat{Y}_s\rangle \geq 0,\ t\leq s\leq T, \  \mbox{P-a.s.}$ \\
If $\beta_2>0,\ \beta_3>0$, then we get
\begin{equation}\label{ee6.5}\begin{array} [c]{llll} \langle
G\hat{X}_t,\hat{Y}_t\rangle &=&E[\langle G\hat{X}_s,\hat{Y}_s\rangle
\mid \mathcal {F}_t]- E[\int_t^s(\langle \Delta
A(r),(\hat{X}_r,\hat{Y}_r,\hat{Z}_r)\rangle+ \int_E\langle G\Delta
h(r,e),\hat{K}_r(e) \rangle \lambda(de))
dr|\mathcal {F}_t]\\
&\geq& \beta_2E[ \int_t^s(|G^T\hat{Y}_r|^2+|G^T\hat{Z}_r|^2)\mid\mathcal {F}_t]+
\beta_3E\int_t^s\int_E|G^T\hat{K}_r(e)|^2\lambda(de)dr\mid\mathcal {F}_t], \ t\leq
s\leq T, \  \mbox{P-a.s.}
\end{array} \end{equation}
Therefore, noticing here $m=1$,\begin{equation}\label{ee6.6}
E[\int_t^s(|\hat{Y}_r|^2+|\hat{Z}_r|^2+
\int_E|\hat{K}_r(e)|^2\lambda(de))dr|\mathcal {F}_t]\leq C\langle
G\hat{X}_t,\hat{Y}_t\rangle,\ t\leq s\leq T, \
\mbox{P-a.s.}\end{equation} Then, from (\ref{ee6.1}) we can get
\begin{equation} \label{ee6.7}E[|\hat{X}_s|^2\mid \mathcal{F}_t] \leq
C|\zeta-\zeta'|^2+ C\langle G\hat{X}_t,\hat{Y}_t\rangle, \ t\leq
s\leq T, \  \mbox{P-a.s.}\end{equation} From (\ref{ee6.3}) we have
\begin{equation} \label{ee6.8}E[|\hat{Y}_s|^2\mid \mathcal
{F}_t]+E[\int_s^T(|\hat{Y}_r|^2+|\hat{Z}_r|^2+
\int_E|\hat{K}_r(e)|^2\lambda(de))dr|\mathcal
{F}_t]
\leq  C|\zeta-\zeta'|^2+ C\langle G\hat{X}_t,\hat{Y}_t\rangle, \
t\leq s\leq T, \  \mbox{P-a.s.}\end{equation} Therefore,
$$\begin{array} [c]{llll} |\hat{Y}_t|^2\leq C|\zeta-\zeta'|^2+
C|\hat{X}_t||\hat{Y}_t|\leq C|\zeta-\zeta'|^2+ C|
\hat{X}_t|^2+{\frac{1}{2}}|\hat{Y}_t|^2,\  \mbox{P-a.s.},
\end{array}$$
which means $|\hat{Y}_t|\leq C|\zeta-\zeta'|,\  \mbox{P-a.s.}$
Then, from (\ref{ee6.7}), (\ref{ee6.8}), we can get
$$\begin{array}{llll} &&E[|\hat{X}_s|^2\mid \mathcal
{F}_t]+E[|\hat{Y}_s|^2\mid \mathcal
{F}_t]+E[\int_s^T(|\hat{Y}_r|^2+|\hat{Z}_r|^2+
\int_E|\hat{K}_r(e)|^2\lambda(de))dr|\mathcal
{F}_t]\\
&\leq & C|\zeta-\zeta'|^2, \  t\leq s\leq T,\
\mbox{P-a.s.}\end{array}$$ If $\beta_2=0,\ \beta_3=0$, then from assumption $(
\mathbf{H2.2})$, we have $\beta_1>0,\
\mu_1>0,\ m=n=1,$ i.e. $G\in \mathbb{R}\setminus\{0\}$.\\
From (\ref{ee6.4}), $$ E[|\hat{X}_T|^2\mid \mathcal
{F}_t]+E[\int_t^T|\hat{X}_r|^2dr|\mathcal {F}_t]\leq CG\hat{X}_t\cdot
\hat{Y}_t,\ C>0.$$ From (\ref{ee6.3}) combined with (\ref{ee6.5}),
$$|\hat{Y}_t|^2+E[\int_t^T(|\hat{Y}_r|^2+|\hat{Z}_r|^2+\int_E|\hat{K}_r(e)|^2\lambda(de))dr|\mathcal
{F}_t]\leq CG\hat{X}_t\cdot \hat{Y}_t\leq
C|\zeta-\zeta'|^2+{\frac{1}{2}}|\hat{Y}_t|^2.$$ Therefore,
$$|\hat{Y}_t|^2+E[\int_t^T(|\hat{Y}_r|^2+|\hat{Z}_r|^2+\int_E|\hat{K}_r(e)|^2\lambda(de))dr|\mathcal
{F}_t]\leq C|\zeta-\zeta'|^2.$$ Furthermore, from (\ref{ee6.1}),$$
E[|\hat{X}_s|^2\mid \mathcal {F}_t]\leq C|\zeta-\zeta'|^2,\ t\leq
s\leq T, \  \mbox{P-a.s.} $$ Therefore,
$$\begin{array} [c]{llll}E[\sup \limits_{t\leq s\leq
T}|\hat{X}_s|^2\mid \mathcal {F}_t]
&\leq& 3|\zeta-\zeta'|^2+ CE[\int_t^T|\Delta b(r)|^2dr+
\int_t^T|\Delta \sigma(r)|^2dr+\int_t^T\int_E|\Delta
h(r,e)|^2\lambda(de)dr|\mathcal
{F}_t] \\
&\leq&
3|\zeta-\zeta'|^2+CE[\int_t^T(|\hat{X}_r|^2+|\hat{Y}_r|^2+|\hat{Z}_r|^2+
\int_E|\hat{K}_r(e)|^2\lambda(de))dr|\mathcal
{F}_t]\\
&\leq& C|\zeta-\zeta'|^2,\  \mbox{P-a.s.};
\end{array} $$
similarly, we have $$\begin{array} [c]{llll}E[\sup \limits_{t\leq
s\leq
T}|\hat{Y}_s|^2\mid \mathcal {F}_t]
&\leq& CE[|\hat{X}_T|^2\mid \mathcal {F}_t]+CE[
\int_t^T(|\hat{X}_r|^2+|\hat{Y}_r|^2+|\hat{Z}_r|^2+
\int_E|\hat{K}_r(e)|^2\lambda(de))dr|\mathcal
{F}_t]\\
&\leq& C|\zeta-\zeta'|^2,\  \mbox{P-a.s.}
\end{array} $$
In this way, we complete the proof of (i). Also, (ii) can be proved
similarly by making full use of the monotonic assumption  $(
\mathbf{H2.2})$. For (iii), similarly, using $( \mathbf{H3.2})$,
$$\begin{array} [c]{llll}&&E[\sup \limits_{t\leq s\leq
t+\delta}|X_s^{t,\zeta}-\zeta|^2\mid \mathcal {F}_t]\\
&\leq& 2E[|\int_t^{t+\delta}
|b(r,X_r^{t,\zeta},Y_r^{t,\zeta},Z_r^{t,\zeta},K_r^{t,\zeta})|dr|^2|\mathcal
{F}_t] + CE[\int_t^{t+\delta}
|\sigma(r,X_r^{t,\zeta},Y_r^{t,\zeta},Z_r^{t,\zeta},K_r^{t,\zeta})|^2dr|\mathcal
{F}_t] \\
&& + CE[\int_t^{t+\delta}\int_E
|h(r,X_{r-}^{t,\zeta},Y_{r-}^{t,\zeta},Z_{r}^{t,\zeta},K_{r}^{t,\zeta},e)|^2\lambda(de)dr|\mathcal
{F}_t]  \\
& \leq & C\delta
E[\int_t^{t+\delta}(1+|X_r^{t,\zeta}|^2+|Y_r^{t,\zeta}|^2+|Z_r^{t,\zeta}|^2+
\int_E|K_r^{t,\zeta}(e)|^2\lambda(de))dr|\mathcal
{F}_t]\\
&&+CE[
\int_t^{t+\delta}(1+|X_r^{t,\zeta}|^2+|Y_r^{t,\zeta}|^2)dr|\mathcal
{F}_t] \\
 & \leq &C\delta E[\sup \limits_{t\leq r\leq
t+\delta}(|X_r^{t,\zeta}|^2+|Y_r^{t,\zeta}|^2)+
\int_t^{t+\delta}(|Z_r^{t,\zeta}|^2+
\int_E|K_r^{t,\zeta}(e)|^2\lambda(de))dr|\mathcal
{F}_t]+C\delta\\
&&+C\delta  E[\sup \limits_{t\leq r\leq
t+\delta}(|X_r^{t,\zeta}|^2+|Y_r^{t,\zeta}|^2)|\mathcal {F}_t]\\
& \leq
&C\delta(1+|\zeta|^2).
\end{array} $$
\end{proof}

\begin{remark}
\label{re6.1} From Proposition \ref{pro6.1}, we have, immediately,
\begin{equation}
|Y_{t}^{t,\zeta}| \leq C(1+|\zeta|);\  \  \  \  \  \  \ |Y_{t}^{t,\zeta}%
-Y_{t}^{t,\zeta^{\prime}}| \leq C|\zeta-\zeta^{\prime}|,\  \  \mbox{P-a.s.},
\end{equation}
where the constant $C>0$ depends only on the Lipschitz constants of
$b,\  \sigma,\ h,\ g \ and \  \Phi.$
\end{remark}

Now we introduce the random field:
\[
u(t,x) = Y_{s}^{t,x}\mid_{s=t}, \  \ (t,x)   \in  [0,T]\times
\mathbb{R}^{n},
\]
where $Y^{t,x}$ is the solution of FBSDE (\ref{equ6.1}) with the initial state
$x\in \mathbb{R}^{n}$.

From Remark \ref{re6.1}, it is easy to check that, for all $t \in
[0,T],\  \mbox{P-a.s.},$
\begin{equation}\label{ee6.10}
\begin{array}
[c]{llll}
&  & \mathrm{(i)}\ |u(t,x)-u(t,y)| \leq C|x-y|, \  \mbox{for} \  \mbox{all}
\ x,y \in \mathbb{R}^{n}; & \\
&  & \mathrm{(ii)} \ |u(t,x)| \leq C(1+|x|), \  \mbox{for} \  \mbox{all}
\ x\in \mathbb{R}^{n}. &
\end{array}
\end{equation}

\begin{remark}
\label{re6.2} Moreover, it is well known that, under the additional assumption
that the functions
\[
b,\  \sigma,\ h,\ g \  \mbox{and} \  \Phi \ are \ deterministic,
\]
 also $u$ is a deterministic function of $(t,x)$.
\end{remark}

The random field $u$ and $Y^{t,\zeta},\ (t,\zeta)\in[0,T]\times L^{2}%
(\Omega,\mathcal{F}_{t},P;\mathbb{R}^{n}),$ are related by the following theorem.

\begin{theorem}
\label{th6.1} Under the assumptions $( \mathbf{H2.1}),\ ( \mathbf{H2.2})$, for
any $t\in[0,T]$ and $\zeta \in L^{2}(\Omega,\mathcal{F}_{t},P;\mathbb{R}^{n}),$
we have
\[
u(t,\zeta)=Y_{t}^{t,\zeta}, \  \  \mbox{P-a.s.}
\]

\end{theorem}

The proof of Theorem \ref{th6.1}  is similar to Theorem A.1 in
\cite{LP} for the decoupled FBSDE with jumps, or Theorem 6.1 in
\cite{BLH}.

\begin{remark}
\label{re6.3} {\rm(i)} From Theorem \ref{th6.1}, obviously, $Y_{s}^{t,\zeta}%
=Y_{s}^{s,X_{s}^{t,\zeta}}=u(s,X_{s}^{t,\zeta}).$

 {\rm(ii)} From now for convenience, we take $\rho(e)=C(1\wedge|e|)$, where $C$ is a constant.
\end{remark}

\begin{proposition}
\label{pro6.2} Under the assumptions $( \mathbf{H2.1}),\ ( \mathbf{H2.2}),\ (
\mathbf{H3.1}),\ ( \mathbf{H3.2})$, for any $p\geq 2$, $0 \leq t \leq T$ and
the associated initial states $\zeta,\  \zeta^{\prime}\in L^p(\Omega,\mathcal{F}%
_{t},P;\mathbb{R}^{n}),$ there exists $\tilde{\delta}_{0}>0$ which
depends on $p$ and the Lipschitz constant and the linear growth
constant $L$, such that
\[%
\begin{array}
[c]{llll}%
{(\mathrm{i)}} &  & E[\mathop{\rm sup}\limits_{t\leq s\leq
t+\delta }|{X}_{s}^{t,\zeta}|^{p}+\mathop{\rm
sup}\limits_{t\leq s\leq
t+\delta}|{Y}_{s}^{t,\zeta}|^{p} + (\int
_{t}^{t+\delta}|{Z}_{s}^{t,\zeta}|^{2}ds)^{\frac{p}{2}}\\
&& + (\int_{t}^{t+\delta}\int
_{E}|{K}_{s}^{t,\zeta}(e)|^{2}\lambda(de)ds)^{\frac{p}{2}}\mid \mathcal{F}_{t}]
\leq
C_{p}(1+|\zeta|^{p}),\  \mbox{P-a.s.} ; & \\
{(\mathrm{ii)}} &  & E[\mathop{\rm sup}\limits_{t\leq s\leq t+\delta}%
|X_{s}^{t,\zeta}-\zeta|^{p}\mid \mathcal{F}_{t}]\leq C_{p}\delta(1 +|\zeta|^{p}),\  \mbox{P-a.s.},\ 0\leq \delta \leq \tilde{\delta}_{0}. &
\end{array}
\]
\end{proposition}

\begin{remark}\label{wrong}
Let us point out that, unlike FBSDEs without jumps, estimates (ii) does not hold true with $\delta^\frac{p}{2}$ instead of $\delta$ at the right hand, that is, one can't get the following estimate like FBSDEs without jumps, even for the decoupled FBSDEs with jumps:
for all $p\geq2$, $E[\sup\limits_{t\leq s\leq t+\delta}|X_{s}^{t,\zeta}-\zeta|^{p}\mid \mathcal{F}_{t}]\leq C_{p}\delta^\frac{p}{2}(1 +|\zeta|^{p}),\  \mbox{P-a.s.},\ 0\leq \delta \leq \delta_{0}$.

Indeed, if the above estimate is true, then one can get, for all $t\leq s\leq s+\delta\leq t+\delta_0$, $$E[
|X_{s+\delta}^{t,\zeta}-X_{s}^{t,\zeta}|^{p}]\leq E[E[|X_{s+\delta}^{s,X_{s}^{t,\zeta}}-X_{s}^{t,\zeta}|^{p}\mid \mathcal{F}_{s}]]\leq C_{p}\delta^\frac{p}{2}E[(1 +|X_{s}^{t,\zeta}|^{p})],$$ and for $\frac{p}{2}>2$,  Kolmogorov's Continuity Criterion would imply the continuity of  the jump process $X^{t,\zeta}$  which is  impossible.

\end{remark}

In order to prove Proposition \ref{pro6.2}, we need the following lemma.
\begin{lemma} \label{inequality} Under the assumptions $( \mathbf{H2.1}),\ ( \mathbf{H2.2}),\ (
\mathbf{H3.1}),\ ( \mathbf{H3.2})$. For any $p\geq2$,
\begin{equation}\label{ee6.1111}
E[(\int_t^{t+\delta}\int_E|K_s^{t,\zeta}(e)|^2\lambda(de)ds)^\frac{p}{2}\mid\mathcal {F}_t]\leq (\frac{p}{2})^\frac{p}{2}E[(\int_t^{t+\delta}\int_E|K_s^{t,\zeta}(e)|^2\mu(dsde))^\frac{p}{2}\mid\mathcal {F}_t].
\end{equation}\end{lemma}
\begin{proof} Setting $f_s:=\int_E|K_s^{t,\zeta}(e)|^2\lambda(de)$, we have
 $$
 \begin{array}{llll}
 &&E[(\int_t^{t+\delta}\int_E|K_s^{t,\zeta}(e)|^2\lambda(de)ds)^\frac{p}{2}\mid\mathcal {F}_t]=E[(\int_t^{t+\delta}f_sds)^\frac{p}{2}\mid\mathcal {F}_t]=\frac{p}{2}E[\int_t^{t+\delta}f_s(\int_t^sf_rdr)^{\frac{p}{2}-1} ds\mid\mathcal {F}_t]\\
 &&=\frac{p}{2}E[\int_t^{t+\delta}\int_E\{(\int_t^sf_rdr)^{\frac{p}{2}-1} \cdot |K_s^{t,\zeta}(e)|^2\}\lambda(de)ds\mid\mathcal {F}_t]\\
  &&=\frac{p}{2}E[\int_t^{t+\delta}\int_E\{(\int_t^sf_rdr)^{\frac{p}{2}-1} \cdot |K_s^{t,\zeta}(e)|^2\}\mu(dsde)\mid\mathcal {F}_t]\\
 &&\leq \frac{p}{2} E[(\int_t^{t+\delta}\int_E|K_s^{t,\zeta}(e)|^2\mu(dsde))(\int_t^{t+\delta}f_rdr)^{\frac{p}{2}-1}\mid\mathcal {F}_t]\\
&&\leq \frac{p}{2} (E[(\int_t^{t+\delta}\int_E|K_s^{t,\zeta}(e)|^2\mu(dsde))^\frac{p}{2}\mid\mathcal {F}_t])^\frac{2}{p}(E[(\int_t^{t+\delta}f_rdr)^{\frac{p}{2}}\mid\mathcal {F}_t])^{1-\frac{2}{p}}.
\end{array}
 $$
Therefore, we have (\ref{ee6.1111})
if $E[(\int_t^{t+\delta}\int_E|K_s^{t,\zeta}(e)|^2\lambda(de)ds)^\frac{p}{2}\mid\mathcal {F}_t]<+\infty,\ \mbox{P-a.s.}$ Otherwise, we approximate $|K_s^{t,\zeta}(e)|^2$ from below by an increasing sequence $K^n$ of non-negative predictable functions over $\Omega\times[0,T]\times E$ such that  $E[(\int_t^{t+\delta}\int_E|K_s^n(e)|^2\lambda(de)ds)^\frac{p}{2}\mid\mathcal {F}_t]<+\infty,\ n\geq 1.$ Then, with the same arguments as above we have
$$E[(\int_t^{t+\delta}\int_E|K_s^n(e)|^2\lambda(de)ds)^\frac{p}{2}\mid\mathcal {F}_t]\leq (\frac{p}{2})^\frac{p}{2} E[(\int_t^{t+\delta}\int_E|K_s^n(e)|^2\mu(dsde))^\frac{p}{2}\mid\mathcal {F}_t],\ n\geq 1, $$
and taking the limit as $n\rightarrow +\infty$ by using the monotone convergence theorem, we obtain (\ref{ee6.1111}).

\end{proof}

Now we give the proof of Proposition \ref{pro6.2}.

\begin{proof} Without loss of generality, we restrict ourselves to the proof for $p=2k,\ k\in \mathbb{Z}^+$.

From the Remarks \ref{re6.1} and \ref{re6.3} we have
$|{Y}_s^{t,\zeta}|=|{Y}_s^{s,{X}_s^{t,\zeta}}|\leq
C(1+|{X}_s^{t,\zeta}|),\  \mbox{P-a.s.}$\\
Since
$${Y}_t^{t,\zeta}={Y}_s^{t,\zeta}+\int_t^s
g(r,{X}_r^{t,\zeta},{Y}_r^{t,\zeta},{Z}_r^{t,\zeta},{K}_r^{t,\zeta})dr-
\int_t^s{Z}_r^{t,\zeta}dB_r-\int_t^s
\int_E{K}_r^{t,\zeta}(e)\tilde{\mu}(drde),\ t\leq s\leq t+\delta,$$
we get from Burkholder-Davis-Gundy inequality and (\ref{ee6.1111}), $$
\begin{array}{llll}
&&E[(\int_t^{t+\delta}|Z_s^{t,\zeta}|^2ds)^{\frac{p}{2}}\mid
\mathcal {F}_t]+E[(\int_t^{t+\delta}
\int_E|K_s^{t,\zeta}(e)|^2\lambda(de)ds)^{\frac{p}{2}}\mid \mathcal
{F}_t]\\
&\leq& E[(\int_t^{t+\delta}|Z_s^{t,\zeta}|^2ds)^{\frac{p}{2}}\mid
\mathcal {F}_t]+C_pE[(\int_t^{t+\delta}
\int_E|K_s^{t,\zeta}(e)|^2\mu(dsde))^{\frac{p}{2}}\mid \mathcal
{F}_t]\\
&\leq& C_pE[\sup \limits_{t\leq s\leq
t+\delta}|\int_t^s{Z}_r^{t,\zeta}dB_r+\int_t^s\int_EK_s^{t,\zeta}(e)\tilde{\mu}(dsde)|^p\mid \mathcal {F}_t]\\
&\leq&C_pE[\sup \limits_{t\leq s\leq
t+\delta}|{Y}_s^{t,\zeta}|^p+(\int_t^{t+\delta}
|g(s,{X}_s^{t,\zeta},{Y}_s^{t,\zeta},{Z}_s^{t,\zeta},{K}_s^{t,\zeta})|ds)^p\mid
\mathcal
{F}_t]\\
&\leq&C_pE[\sup \limits_{t\leq s\leq
t+\delta}|{Y}_s^{t,\zeta}|^p\mid \mathcal
{F}_t]+C_pE[(\int_t^{t+\delta}(1+|{X}_s^{t,\zeta}|+|{Y}_s^{t,\zeta}|+|{Z}_s^{t,\zeta}|+
\int_E|{K}_s^{t,\zeta}(e)|(1\wedge|e|)\lambda(de))ds)^p\mid \mathcal
{F}_t]\\
&\leq&C_pE[\sup \limits_{t\leq s\leq
t+\delta}|{Y}_s^{t,\zeta}|^p\mid \mathcal
{F}_t]+C_p\delta^p+C_pE[\sup \limits_{t\leq s\leq
t+\delta}|{X}_s^{t,\zeta}|^p+\sup \limits_{t\leq s\leq
t+\delta}|{Y}_s^{t,\zeta}|^p\mid \mathcal
{F}_t]\delta^p\\
&&+C_pE[(\int_t^{t+\delta}|Z_s^{t,\zeta}|^2ds)^{\frac{p}{2}}\mid
\mathcal {F}_t]\delta^{\frac{p}{2}}+C_pE[( \int_t^{t+\delta}
\int_E|K_s^{t,\zeta}(e)|^2\lambda(de)ds)^{\frac{p}{2}}\mid \mathcal
{F}_t]\delta^{\frac{p}{2}}\\
&= &C_p\delta^p+C_p\delta^pE[\sup \limits_{t\leq s\leq
t+\delta}|{X}_s^{t,\zeta}|^p\mid \mathcal
{F}_t]+(C_p+C_p\delta^p)E[\sup \limits_{t\leq s\leq
t+\delta}|{Y}_s^{t,\zeta}|^p\mid \mathcal
{F}_t]\\
&&+C_p\delta^{\frac{p}{2}}E[(
\int_t^{t+\delta}|Z_s^{t,\zeta}|^2ds)^{\frac{p}{2}}\mid \mathcal
{F}_t]+C_p\delta^{\frac{p}{2}}E[( \int_t^{t+\delta}
\int_E|K_s^{t,\zeta}(e)|^2\lambda(de)ds)^{\frac{p}{2}}\mid \mathcal
{F}_t].
\end{array}
$$
Choosing $\delta_0>0,$ such that $1-C_p\delta_0^{\frac{p}{2}}>0,$
 we get, for any $0\leq\delta \leq \delta_0$,
\begin{equation}\label{ee6.669}\begin{array}{llll}
&&E[(\int_t^{t+\delta}|Z_s^{t,\zeta}|^2ds)^{\frac{p}{2}}\mid
\mathcal {F}_t]+E[(\int_t^{t+\delta}
\int_E|K_s^{t,\zeta}(e)|^2\lambda(de)ds)^{\frac{p}{2}}\mid \mathcal
{F}_t]\\
&\leq& C_p\delta^{p}+C_p\delta^{p}E[\sup \limits_{t\leq s\leq
t+\delta}|{X}_s^{t,\zeta}|^p\mid \mathcal
{F}_t]+(C_p+C_p\delta^p)E[\sup \limits_{t\leq s\leq
t+\delta}|{Y}_s^{t,\zeta}|^p\mid \mathcal
{F}_t].\end{array}\end{equation} On the other hand, from Remark
\ref{re6.3} and (\ref{ee6.10}), for $t\leq s\leq T$, \begin{equation}\label{ee6.0000}
\begin{array}{llll}&&
E[\sup \limits_{t\leq r\leq
s}|X_r^{t,\zeta}-\zeta|^p\mid \mathcal {F}_t]\\
&\leq &C_pE[(\int_t^s
|b(r,X_r^{t,\zeta},Y_r^{t,\zeta},Z_r^{t,\zeta},K_r^{t,\zeta})|dr)^p\mid
\mathcal {F}_t]+C_pE[(\int_t^s|
\sigma(r,X_r^{t,\zeta},Y_r^{t,\zeta},Z_r^{t,\zeta},K_r^{t,\zeta})|^2dr)^{\frac{p}{2}}\mid
\mathcal
{F}_t]\\
&&+C_pE[(\int_t^s\int_E|
h(r,X_{r-}^{t,\zeta},Y_{r-}^{t,\zeta},Z_{r}^{t,\zeta},K_{r}^{t,\zeta}(e),e)|^2\mu(drde))^{\frac{p}{2}}\mid
\mathcal
{F}_t]\\
&\leq
&C_pE[(\int_t^s(1+|X_r^{t,\zeta}-\zeta|+|\zeta|+|Z_r^{t,\zeta}|+
\int_E|K_r^{t,\zeta}(e)|\lambda(de))dr)^p\mid \mathcal
{F}_t]\\
&&+C_pE[(\int_t^s(1+|X_r^{t,\zeta}|+|Y_r^{t,\zeta}|)^2dr)^{\frac{p}{2}}\mid \mathcal
{F}_t]\\
&&+C_pE[(\int_t^s\int_E|
h(r,X_{r-}^{t,\zeta},Y_{r-}^{t,\zeta},Z_{r}^{t,\zeta},K_{r}^{t,\zeta}(e),e)|^2\mu(drde))^{\frac{p}{2}}\mid
\mathcal
{F}_t]\\
&\leq& C_p(1+|\zeta|^p)(s-t)^{\frac{p}{2}}+C_p(s-t)^{\frac{p}{2}}E[(\int_t^{s}|Z_r^{t,\zeta}|^2dr)^{\frac{p}{2}}+(
\int_t^{s}
\int_E|K_r^{t,\zeta}(e)|^2\lambda(de)dr)^{\frac{p}{2}}\mid \mathcal
{F}_t]\\
&&+C_pE[\int_t^s|X_r^{t,\zeta}-\zeta|^pdr\mid
\mathcal
{F}_t]+C_pE[(\int_t^s\int_E|
h(r,X_{r-}^{t,\zeta},Y_{r-}^{t,\zeta},Z_{r}^{t,\zeta},K_{r}^{t,\zeta}(e),e)|^2\mu(drde))^{\frac{p}{2}}\mid
\mathcal
{F}_t],
\end{array}
\end{equation}
where
\begin{equation}\label{ee6.30}\begin{array}{llll}&&
E [(\int_t^{s}\int_E |h(r,X^{t,\zeta}_{r-} , Y^{t,\zeta}_{r-}
,Z^{t,\zeta}_r,K^{t,\zeta}_r,e)|^2\mu(drde))^{\frac{p}{2}} |
\mathcal {F}_t]\\
&\leq&  E [(\int_t^{t+\delta}\int_EC(1\wedge |e|^2)(1+|X^{t,\zeta}_{r-}|+|Y^{t,\zeta}_{r-}|)^2\mu(drde))^{\frac{p}{2}} |
\mathcal {F}_t]\\
&\leq& C_p E [(\int_t^{t+\delta}\int_E(1\wedge |e|^2)(1+|X^{t,\zeta}_{r-}|^2+|Y^{t,\zeta}_{r-}|^2)\mu(drde))^{\frac{p}{2}} |
\mathcal {F}_t]\\
&\leq& C_p E [(\int_t^{t+\delta}\int_E(1\wedge |e|^2)(1+|X^{t,\zeta}_{r-}-\zeta|^2+|\zeta|^2)\mu(drde))^{\frac{p}{2}} |
\mathcal {F}_t]\\
&\leq& C_p E [(\int_t^{t+\delta}\int_E(1\wedge |e|^2)\mu(drde))^{\frac{p}{2}} |
\mathcal {F}_t](1+|\zeta|^p)+C_pE[(\int_t^{t+\delta}\int_E(1\wedge |e|^2)|X^{t,\zeta}_{r-}-\zeta|^2\mu(drde))^{\frac{p}{2}} |\mathcal {F}_t].
\end{array}\end{equation}
Notice that  \begin{equation}\label{ee6.11111}E[(\int_t^{t+\delta}\int_E(1\wedge |e|^2)|X^{t,\zeta}_{r-}-\zeta|^2\mu(drde))^{\frac{p}{2}} |\mathcal {F}_t]\leq C_p\delta E[\sup\limits_{t\leq s\leq t+\delta}|X^{t,\zeta}_{r}-\zeta|^p|\mathcal {F}_t].\end{equation} Indeed, we denote $\tilde{X}^{t,\zeta}_{s-}:=X^{t,\zeta}_{s-}-\zeta, \ \rho_s(e):=(1\wedge |e|^2)|\tilde{X}^{t,\zeta}_{s}|^2,\ \bar{\rho}_s(e):=|\tilde{X}^{t,\zeta}_{s}|^2,\ A_r:=\int_t^r\int_E \rho_s(e)\mu(dsde)$. Then, from Young inequality we have
$$\begin{array}{llll}
&&A_r^p-A_{t-}^p=\sum\limits_{t\leq s\leq r}(A_s^p-A_{s-}^p)=\sum\limits_{t\leq s\leq r}((\int_t^s\int_E \rho_r(e')\mu(drde'))^p-(\int_t^{s-}\int_E \rho_r(e')\mu(drde'))^p)\\
&=& \sum\limits_{t\leq s\leq r}\int_E ((\int_t^{s-}\int_E \rho_r(e')\mu(drde')+\rho_s(e))^p-(\int_t^{s-}\int_E \rho_r(e')\mu(dsde'))^p)\mu(\{s\},de)\\
&=&  \sum\limits_{t\leq s\leq r}\int_E \sum\limits_{l=1}^p\left(
\begin{array}
[c]{c}
p\\
l
\end{array}
\right)(\int_t^{s-}\int_E \rho_r(e')\mu(drde'))^{p-l}\rho_s(e)^l\mu(\{s\},de)\\
&=&  \int_t^r\int_E \sum\limits_{l=1}^p\left(
\begin{array}
[c]{c}
p\\
l
\end{array}
\right)(\int_t^{s-}\int_E \rho_r(e')\mu(drde'))^{p-l}\rho_s(e)^l\mu(dsde)\\
&\leq& C_p \int_t^r\int_E(1\wedge |e|^2)((\int_t^{s-}\int_E \rho_r(e')\mu(drde'))^{p}+|\bar{\rho}_s(e)|^p)\mu(dsde).
\end{array}$$
Therefore, $$E[|\int_t^r\int_E \rho_s(e)\mu(dsde)|^p|\mathcal {F}_t]\leq C_p E[\int_t^r\int_E(1\wedge |e|^2)((\int_t^{s-}\int_E \rho_r(e')\mu(drde'))^{p}+|\bar{\rho}_s(e)|^p)\lambda(de)ds|\mathcal {F}_t].$$
From the Gronwall inequality, we get
$$E[|\int_t^r\int_E \rho_s(e)\mu(dsde)|^p|\mathcal {F}_t]\leq C_pE[ \int_t^r\int_E(1\wedge |e|^2)|\bar{\rho}_s(e)|^p\lambda(de)ds|\mathcal {F}_t].$$
Therefore, $$E[|\int_t^r\int_E|X_{s-}^{t,\zeta}-\zeta|^{2}(1\wedge |e|^2)\mu(dsde)|^p|\mathcal {F}_t]\leq C_p E[\int_t^r|X_{s-}^{t,\zeta}-\zeta|^{2p}ds|\mathcal {F}_t]\leq C_p(r-t)E[\sup\limits_{t\leq s\leq r}|X_{s-}^{t,\zeta}-\zeta|^{2p}|\mathcal {F}_t].$$
Similarly,
 $E[|\int_t^r\int_E(1\wedge |e|^2)\mu(dsde)|^\frac{p}{2}|\mathcal {F}_t]\leq C_p(r-t).$
 Thus, from (\ref{ee6.30}) we have
 \begin{equation}
\begin{array}{llll}&&E[(\int_t^s\int_E|h(r,X_{r-}^{t,\zeta},Y_{r-}^{t,\zeta},Z_{r}^{t,\zeta},K_{r}^{t,\zeta},e)|^2\mu(drde))^{\frac{p}{2}}\mid \mathcal {F}_t]\leq
C_p\delta(1+|\zeta|^p)+C_p\delta E[\sup \limits_{t\leq r\leq
s}|X_r^{t,\zeta}-\zeta|^p\mid \mathcal {F}_t].\end{array}\end{equation}
Consequently, from (\ref{ee6.0000}),
  \begin{equation}\label{ee6.9}
\begin{array}{llll}E[\sup \limits_{t\leq r\leq
t+\delta}|X_r^{t,\zeta}-\zeta|^p\mid \mathcal {F}_t]
&\leq&
C_p\delta(1+|\zeta|^p)+C_p\delta E[\sup \limits_{t\leq r\leq
t+\delta}|X_r^{t,\zeta}-\zeta|^p\mid \mathcal {F}_t]\\
&&+C_p\delta^{\frac{p}{2}}E[(\int_t^{t+\delta}|Z_r^{t,\zeta}|^2dr)^{\frac{p}{2}}+(
\int_t^{t+\delta}
\int_E|K_r^{t,\zeta}(e)|^2\lambda(de)dr)^{\frac{p}{2}}\mid \mathcal
{F}_t],\ \mbox{P-a.s.}\end{array}\end{equation}
Choosing $\delta_1>0$, such that $1-C_p\delta_1>0$, for any $0\leq\delta\leq\delta_1$, we have
\begin{equation}\label{ee6.9990}
\begin{array}{llll}&&E[\sup \limits_{t\leq r\leq
t+\delta}|X_r^{t,\zeta}-\zeta|^p\mid \mathcal {F}_t]\\
&\leq&
C_p\delta(1+|\zeta|^p)+C_p\delta^{\frac{p}{2}}E[(\int_t^{t+\delta}|Z_r^{t,\zeta}|^2dr)^{\frac{p}{2}}+(
\int_t^{t+\delta}
\int_E|K_r^{t,\zeta}(e)|^2\lambda(de)dr)^{\frac{p}{2}}\mid \mathcal
{F}_t],\ \mbox{P-a.s.}\end{array}\end{equation}
Then, from (\ref{ee6.669}), (\ref{ee6.9990}) and $|Y_s^{t,\zeta}|\leq
C(1+|X_s^{t,\zeta}|)$, we have
$$\begin{array}{llll}&&
E[(\int_t^{t+\delta}|Z_s^{t,\zeta}|^2ds)^{\frac{p}{2}}\mid \mathcal
{F}_t]+E[(\int_t^{t+\delta}
\int_E|K_s^{t,\zeta}(e)|^2\lambda(de)ds)^{\frac{p}{2}}\mid \mathcal
{F}_t]\\
&\leq& C_p\delta^p(1+|\zeta|^p)+ C_p\delta^p+C_p+(C_p+
C_p\delta^p)E[\sup \limits_{t\leq s\leq
t+\delta}|X_s^{t,\zeta}-\zeta|^p\mid \mathcal {F}_t]\\
&\leq& C_p\delta^p+C_p+C_p\delta(1+|\zeta|^p)\\
&&+(C_p+
C_p\delta^p)C_p\delta^{\frac{p}{2}}(E[(\int_t^{s}|Z_r^{t,\zeta}|^2dr)^{\frac{p}{2}}\mid
\mathcal {F}_t]+E[(\int_t^{t+\delta}
\int_E|K_s^{t,\zeta}(e)|^2\lambda(de)ds)^{\frac{p}{2}}\mid \mathcal
{F}_t]),
\end{array}$$
and taking $0<\tilde{\delta}_0\leq \min(\delta_0,\delta_1)$ such that $1-(C_p+
C_p\tilde{\delta}_0^p)C_p\tilde{\delta}_0^{\frac{p}{2}}>0$, we have for all $0\leq
\delta \leq \tilde{\delta}_0,$ $$ E[(
\int_t^{t+\delta}|Z_s^{t,\zeta}|^2ds)^{\frac{p}{2}}\mid \mathcal
{F}_t]+E[(\int_t^{t+\delta}
\int_E|K_s^{t,\zeta}(e)|^2\lambda(de)ds)^{\frac{p}{2}}\mid \mathcal
{F}_t]\leq  C_p(1+|\zeta|^p),\  \mbox{P-a.s.} $$ From
(\ref{ee6.9990}), we get
$$E[\sup \limits_{t\leq s\leq
t+\delta}|X_s^{t,\zeta}-\zeta|^p\mid \mathcal {F}_t]\leq
C_p\delta(1 +|\zeta|^p),\  \mbox{P-a.s.},\ 0\leq
\delta \leq \tilde{\delta}_0.$$ Hence, finally, from
$|{Y}_s^{t,\zeta}|\leq C(1+|{X}_s^{t,\zeta}|)$, we have $$E[\sup
\limits_{t\leq s\leq t+\delta}|Y_s^{t,\zeta}|^p\mid \mathcal
{F}_t]\leq C_p(1 +|\zeta|^p),\  \mbox{P-a.s.},\  0\leq \delta \leq
\tilde{\delta}_0.$$\end{proof}

\subsection{{\protect \large {Well-posedness and regularity results of fully
coupled FBSDEs with jumps on the small time interval}}}

In this subsection,  we first prove that the fully
coupled FBSDEs with jumps have a unique solution on a small time
interval, if the Lipschitz coefficients of $\sigma,\ h$ with respect
to $z,\ k$ are sufficiently small. Then, under these assumptions, we
prove some regularity results for the solutions of fully coupled
FBSDEs with jumps.
\begin{theorem}
\label{pro6.3} We suppose the assumptions $( \mathbf{H2.1}),\ ( \mathbf{H3.1}%
),\ ( \mathbf{H3.3})$ hold true, where assumption $( \mathbf{H3.3})$
is the following:
\begin{description}
\item[$( \mathbf{H3.3})$] The Lipschitz constant $L_{\sigma}\geq0$ of $\sigma$
with respect to $z,\ k$ is sufficiently small, i.e., there exists
some
$L_{\sigma}\geq0$ small enough such that, for all $t\in[0,T],\ x_{1},x_{2}%
\in \mathbb{R}^{n},\ y_{1},y_{2}\in \mathbb{R},\ z_{1},z_{2}\in \mathbb{R}%
^{d},\ k_{1},k_{2}\in \mathbb{R}$,
\[
|\sigma(t,x_{1},y_{1},z_{1},k_{1})-\sigma(t,x_{2},y_{2},z_{2},k_{2})|\leq
K(|x_{1}-x_{2}|+|y_{1}-y_{2}|)+L_{\sigma}(|z_{1}-z_{2}|+|k_{1}-k_{2}|).
\]
Also the Lipschitz coefficient $L_{h}(\cdot)$ of $h$ with respect to
$z,\ k$ is sufficiently small, i.e., there exists a  function
$L_{h}:E\rightarrow \mathbb{R}^{+}$ with $\tilde{C}_h:=\max(\sup\limits_{e\in E}L^2_h(e),\int_{E}L^{2}
_{h}(e)\lambda(de))<+\infty$ sufficiently small, and for all $t\in[0,T],\ x_{1},x_{2}\in
\mathbb{R}^{n},\ y_{1},y_{2}\in \mathbb{R},\ z_{1},z_{2}\in \mathbb{R}%
^{d},\ k_{1},k_{2}\in \mathbb{R}, \ e\in E$,
\[
|h(t,x_{1},y_{1},z_{1},k_{1},e)-h(t,x_{2},y_{2},z_{2},k_{2},e)|\leq
\rho(e)(|x_{1}-x_{2}|+|y_{1}-y_{2}|)+L_{h}(e)(|z_{1}-z_{2}|+|k_{1}-k_{2}|).
\]
\end{description}
Then, there exists a constant $\delta_{0}>0$ only depending on the
Lipschitz constants $K$ and $L_\sigma,\ \tilde{C}_h$, such that, for
every $0\leq \delta \leq \delta_{0}$, and $\zeta \in
L^{2}(\Omega,\mathcal{F}_{t},P;\mathbb{R}^{n})$, FBSDE
(\ref{equ6.1}) has a unique solution
$(\Pi_{s}^{t,\zeta},K_{s}^{t,\zeta})_{s\in[t,t+\delta]}$ on the time
interval $[t, t +\delta]$.
\end{theorem}

\begin{proof}  It is easy to see, for any $v = ((y, z),k)\in \mathcal {M}^2(t, T;\mathbb{R}^{1+d})\times
K_\lambda^2(t, T;\mathbb{R})$, there exists a unique solution $V =
((Y,Z),K) \in \mathcal {M}^2(t, T;\mathbb{R}^{1+d})\times
K_\lambda^2(t, T;\mathbb{R})$ to the following decoupled FBSDE with
jumps:
\begin{equation}\label{equ6.2}\left \{ \begin{array}{llll}
dX_s &=& b(s,X_s, y_s, z_s,k_s)ds + \sigma(s,X_s, y_s,
z_s,k_s)dB_s+\int_Eh(s,X_{s-}, y_{s-}, z_{s},k_{s}(e),e)\tilde{\mu}(dsde),\\
dY_s &=&-g(s,X_s, Y_s,Z_s,K_s)ds + Z_sdB_s+\int_EK_s(e)\tilde{\mu}(dsde),\  s \in [t, T], \\
X_t& =& \zeta,\\
Y_T &=& \Phi(X_T). \end{array}\right.\end{equation} We will prove
that there exists a constant $ \delta_0>0,$ only depending on the
Lipschitz constants $K$, $L_\sigma$ and $L_h(\cdot)$ such that for
every $0\leq \delta \leq \delta_0$ the following mapping $$I:
\mathcal {M}^2(t, t + \delta;\mathbb{R}^{1+d})\times K_\lambda^2(t,
t + \delta;\mathbb{R}) \rightarrow \mathcal {M}^2(t, t +
\delta;\mathbb{R}^{1+d})\times K_\lambda^2(t, t +
\delta;\mathbb{R})$$ is a contraction. Let $v_i = ((y_i,
z_i),k_i)\in \mathcal {M}^2(t, t + \delta;\mathbb{R}^{1+d})\times
K_\lambda^2(t, t + \delta;\mathbb{R})$, and $V_i = I(v_i), \ i = 1,
2.$ We define $\hat{v} = ((y_1 -y_2, z_1 - z_2),k_1-k_2)$, and
$\hat{V} = ((Y_1 -Y_2,Z_1-Z_2),K_1-K_2),\  \hat{X}= X_1-X_2.$ Then,
by the usual techniques and the Gronwall inequality, we get
\begin{equation}\label{ee6.11}\begin{array}{llll}
&&E[\sup \limits_{t\leq
s\leq T} |\hat{X}_s|^2 | \mathcal {F}_t]\\
&\leq& CE[
\int^T_t|\hat{y}_s|^2ds |\mathcal {F}_t] + C((T-t) +L_\sigma^2+\int_EL^2_h(e)\lambda(de)+\sup \limits_{e\in E}L^2_h(e))E[\int^T_t(|\hat{z}_s|^2+\int_E|\hat{k}_s(e)|^2\lambda(de))ds|\mathcal{F}_t] \\
&\leq& C(T-t)E[\sup \limits_{t\leq s\leq
T}|\hat{y}_s|^2|\mathcal{F}_t] + C((T-t) + L_\sigma^2+\tilde{C}_h)E[ \int^T_t(|\hat{z}_s|^2+
\int_E|\hat{k}_s(e)|^2\lambda(de))ds|\mathcal{F}_t].\end{array}
\end{equation}

On the other hand, by using BSDE standard estimate, combined with
(\ref{ee6.11}), we get
$$\begin{array}{llll}&&E[\sup \limits_{t\leq s\leq T} |\hat{Y}_s|^2 +
\int^T_t |\hat{Z}_s|^2ds+\int^T_t\int_E|\hat{K}_s(e)|^2\lambda(de)ds]\\
&\leq & CE[|\Phi(X^1_T )-\Phi(X^2_T)|^2] + CE[\int^T_t
|g(r,X^1_r,V^1_r )-g(r,X^2_r,V^1_r)|^2dr]\\
&\leq& CE[|\hat{X}_T|^2] +
CE[\int^T_t|\hat{X}_r|^2dr]\\
&\leq& C(T-t)E[\sup \limits_{t\leq s\leq T} |\hat{y}_s|^2] + C((T- t)
+ L^2_\sigma+\tilde{C}_h)E[\int^T_t(|\hat{z}_s|^2+\int_E|\hat{k}_s(e)|^2\lambda(de))ds] \\
&\leq&C((T-t) + L^2_\sigma+ \tilde{C}_h)(E[\sup
\limits_{t\leq s\leq T} |\hat{y}_s|^2] + E[ \int^T_t(|\hat{z}_s|^2+
\int_E|\hat{k}_s(e)|^2\lambda(de))ds]).
\end{array}$$
As $L_\sigma,\ \tilde{C}_h$ are sufficiently small, there exists
$\delta_0>0$ such that $C\delta_0 + CL^2_\sigma+C\tilde{C}_h<{\frac{1}{2}}$, and therefore, for any $0\leq \delta \leq
\delta_0$, we have
\begin{equation}\begin{array}{llll}&&E[\sup \limits_{t\leq s\leq t+\delta}
|\hat{Y}_s|^2 +
\int_t^{t+\delta}|\hat{Z}_s|^2ds+\int^T_t\int_E|\hat{K}_s(e)|^2\lambda(de)ds]\\
&&\leq{\frac{1}{2}}(E[\sup \limits_{t\leq s\leq t+\delta}
|\hat{y}_s|^2+ \int_t^ {t+\delta}|\hat{z}_s|^2ds+ \int^T_t
\int_E|\hat{k}_s(e)|^2\lambda(de)ds]),\end{array}\end{equation}
which means, for any $0\leq \delta \leq \delta_0$ this mapping $I$
has a unique fixed point $I(V ) = V$, i.e., FBSDE (\ref{equ6.1}) has
a unique solution
$(\Lambda^{t,\zeta}_s,K^{t,\zeta}_s)_{s\in[t,t+\delta]}:=(X^{t,\zeta}_s,Y^{t,\zeta}_s,Z^{t,\zeta}_s,K^{t,\zeta}_s)_{s\in[t,t+\delta]}$
on $[t, t + \delta]$.

\end{proof}

\begin{remark}
\label{re6.4} In fact, from the proof we see that $L_{\sigma},\
\tilde{C}_h\geq0$ such that $CL^{2}_{\sigma}+C\tilde{C}_h<1$ is sufficient for Proposition
\ref{pro6.2}.
\end{remark}

Next we will prove a comparison theorem for the following fully
coupled FBSDE with jumps:
\begin{equation}
\label{equ6.4}\left \{
\begin{array}
[c]{llll}%
dX_{s} & = & b(s,X_{s}, Y_{s}, Z_{s})ds + \sigma(s,X_{s}, Y_{s}, Z_{s}%
)dB_{s}+\int_{E}h(s,X_{s-}, Y_{s-}, Z_{s},e)\tilde{\mu
}(dsde), & \\
dY_{s} & = & -g(s,X_{s}, Y_{s},Z_{s},K_{s})ds + Z_{s}dB_{s}+\int
_{E}K_{s}(e)\tilde{\mu}(dsde),\ s \in[t, t+\delta], & \\
X_{t} & = & \zeta, & \\
Y_{ t+\delta} & = & \Phi(X_{ t+\delta}). &
\end{array}
\right.
\end{equation}

\begin{theorem}
\label{th6.2} (Generalized Comparison Theorem) We suppose that the
assumptions $( \mathbf{H2.1}),\ ( \mathbf{H3.1}),\ ( \mathbf{H3.3})$
are satisfied. Let $\delta_{0} > 0$ be a constant, only depending on
the Lipschitz constants $K$, $L_\sigma$ and $L_h(\cdot)$,
such that for every $0\leq \delta \leq \delta_{0}$ and $\zeta \in L^{2}%
(\Omega,\mathcal{F}_{t}, P;\mathbb{R}^{n})$, FBSDE (\ref{equ6.4})
has a unique solution
$(X^{i}_{s},Y^{i}_{s},Z^{i}_{s},K^{i}_{s})_{s\in[t,t+\delta]}$
associated with $(b,\sigma,g, \zeta,\Phi_{i})$ on the time interval
$[t, t + \delta]$, respectively. Then, if for any $0\leq \delta \leq
\delta_{0}$ it holds
$\Phi_{1}(X_{t+\delta}^2)\geq \Phi_{2}(X_{t+\delta}^2),\ \mbox{P-a.s.}$ (resp., $\Phi_{1}(X_{t+\delta}^1)\geq \Phi_{2}(X_{t+\delta}^1),\ \mbox{P-a.s.}$), we also get $Y^{1}_{t}\geq Y^{2}_{t}%
$,\ \mbox{P-a.s.}
\end{theorem}

The proof is similar to that of Theorem 4.1 in Wu \cite{W2003}; we
sketch it. For notational simplification, we assume $d = n = 1$.
\begin{proof} We define $(\hat{X},\hat{Y},\hat{Z},\hat{K}) :=(
X^1-X^2,\ Y^1-Y^2, Z^1-Z^2, K^1-K^2).$ Then $(\hat{X},
\hat{Y},\hat{Z},\hat{K})$ satisfies the following FBSDE:
\begin{equation}\label{equ6.6}\left \{ \begin{array}{llll}d\hat{X}_s &=& (b^1_s\hat{X}_s + b^2_s\hat{Y}_s + b^3_s\hat{Z}_s)ds + (\sigma^1_s\hat{X}_s + \sigma^2_s\hat{Y}_s+\sigma^3_s\hat{Z}_s)dB_s+\int_E(h^1_s\hat{X}_{s-} + h^2_s\hat{Y}_{s-} + h^3_s\hat{Z}_s)\tilde{\mu}(dsde),\\
d\hat{Y}_s &=& -(g^{1}_s\hat{X}_s + g^{2}_s\hat{Y}_s + g^{3}_s\hat{Z}_s+ g^{4}_s\int_E\hat{K}_s(e)l(e)\lambda(de))ds+\hat{Z}_sdB_s+\int_E\hat{K}_s(e)\tilde{\mu}(dsde), \\
\hat{X}_t& =& 0,\\
\hat{Y}_{t+\delta}&=& \bar{\Phi}\hat{X}_{t+\delta}+
\Phi_1(X^2_{t+\delta})-\Phi_2(X^2_{t+\delta}),
\end{array}\right.\end{equation} where
$$\begin{array}{llll}
&&\bar{\Phi
}=\begin{cases}\frac{\Phi^1(X^1_{t+\delta})-\Phi^1(X^2_{t+\delta})}{X^1_{t+\delta}-X^2_{t+\delta}},\qquad \quad \    \  \hat{X}_{t+\delta}\neq0,\\
0, \hskip3.7cm  \mbox{otherwise};\end{cases}\
l^1_s =\begin{cases}
\frac{l(s,X^1_s,Y^1_s,Z^1_s)-l(s,X^2_s,Y^1_s,Z^1_s)}{X^1_s-X^2_s},\  \ \ \  \hat{X}_s\neq 0,\\
0, \hskip3,7cm  \mbox{otherwise};\end{cases}\\
&&l^2_s =\begin{cases} \frac{l(s,X^2_s,Y^1_s,Z^1_s)-l(s,X^2_s,Y^2_s,Z^1_s)}{Y^1_s-Y^2_s},\  \ \ \ \hat{Y}_s \neq 0,\\
0,\hskip3.7cm  \mbox{otherwise};\end{cases}\
l^3_s =\begin{cases} \frac{l(s,X^2_s,Y^2_s,Z^1_s)-l(s,X^2_s,Y^2_s,Z^2_s)}{Z^1_s -Z^2_s},\  \ \ \  \hat{Z}_s \neq0,\\
0, \hskip3.7cm  \mbox{otherwise};\end{cases}\\
\end{array}$$
where $l(\cdot) = b(\cdot),\ \sigma(\cdot),\ h(\cdot,e),$ respectively, when $l=h$, in the above
representation, $X^1_s,\ X^2_s,\ Y^1_s,\ Y^2_s$ become  $X^1_{s-},\
X^2_{s-},\ Y^1_{s-},\ Y^2_{s-}$, respectively, and
$$\begin{array}{llll}
&&g^{1}_s =\begin{cases}
\frac{g(s,X^1_s,Y^1_s,Z^1_s,K^1_s)-g(s,X^2_s,Y^1_s,Z^1_s,K^1_s)}{X^1_s-X^2_s},\ \ \ \hat{X}_s\neq 0,\\
0, \hskip4.7cm \mbox{otherwise};\end{cases}\
g^{2}_s =\begin{cases} \frac{g(s,X^2_s,Y^1_s,Z^1_s,K^1_s)-g(s,X^2_s,Y^2_s,Z^1_s,K^1_s)}{Y^1_s-Y^2_s},\ \ \ \hat{Y}_s \neq 0,\\
0,\hskip4.7cm \mbox{otherwise};\end{cases}\\
&&g^{3}_s =\begin{cases} \frac{g(s,X^2_s,Y^2_s,Z^1_s,K^1_s)-g(s,X^2_s,Y^2_s,Z^2_s,K^1_s)}{Z^1_s -Z^2_s},\ \ \ \hat{Z}_s \neq0,\\
0, \hskip4.7cm \mbox{otherwise};\end{cases}\\
&&g^{4}_s =\begin{cases} \frac{g(s,X^2_{s},Y^2_{s},Z^2_s,K^1_s)-g(s,X^2_{s},Y^2_{s},Z^2_s,K^2_s)}{\int_EK^1_s(e)l(e)\lambda(de) -\int_EK^2_s(e)l(e)\lambda(de)},\  \int_E\hat{K}_s(e)l(e)\lambda(de) \neq0,\\
0, \hskip4.7cm \mbox{otherwise}.\end{cases}\\
\end{array}$$
It's easy to check that (\ref{equ6.6}) satisfies $( \mathbf{H2.1}),\ ( \mathbf{H3.1}),\ ( \mathbf{H3.3})$. Therefore, from Proposition
\ref{pro6.3}, there exists a constant $0 < \delta_1\leq \delta_0,$
such that for every $0 \leq\delta \leq \delta_1$, (\ref{equ6.6}) has a
unique solution $(\hat{X} , \hat{Y} , \hat{Z}, \hat{K})$ on $[t,
t+\delta]$. Now we want to prove $\hat{Y}_t\geq0.$ For this, we introduce
the dual FBSDE with jumps
\begin{equation}\label{equ6.7}\left \{ \begin{array}{llll}dP_s& =& (g^{2}_sP_s -b^2_sQ_s-
\sigma^2_sM_s- h^2_sN_s)ds +
(g^{3}_sP_s -b^3_s Q_s- \sigma^3_sM_s- h^3_sN_s)dB_s+\int_Eg_s^4P_{s-}l(e)\widetilde{\mu}(dsde), \\
dQ_s&=& (g^{1}_s P_s -b^1_s Q_s- \sigma^1_sM_s- h^1_sN_s)ds+M_sdB_s+\int_EN_{s-}(e)\widetilde{\mu}(dsde),\\
P_t &=& 1, \\
Q_{t+\delta}& = &-\bar{\Phi}P_{t+\delta}.\end{array}\right.
\end{equation} Notice that also  (\ref{equ6.7}) satisfies $( \mathbf{H2.1}),\ ( \mathbf{H3.1}),\ ( \mathbf{H3.3})$. Consequently, due to
Theorem \ref{pro6.3}, there exists a constant $0 < \delta_2 \leq
\delta_1$, such that for every $0\leq \delta \leq \delta_2$,
(\ref{equ6.7}) has a unique solution $(P,Q,M,N)$ on $[t, t +
\delta]$. Applying It\^{o}'s formula to $\hat{X}_sQ_s+
\hat{Y}_sP_s$, we deduce from the equations (\ref{equ6.6}) and
(\ref{equ6.7}),
$$\hat{Y}_t = E[(\Phi_1(X^2_{t+\delta}) -
\Phi_2(X^2_{t+\delta}))P_{t+\delta}|\mathcal {F}_t]. $$ Since
$\Phi_1(X^2_{t+\delta})\geq \Phi_2(X^2_{t+\delta}),$ \mbox{P-a.s.},
if we can prove $P_{t+\delta}\geq 0,$ \mbox{P-a.s.}, then we  get
$\hat{Y}_t \geq0,$ \mbox{P-a.s.} For this we define the following
stopping time: $\tau = \inf \{s > t : P_s \leq0\} \wedge ( t +
\delta ) $. So, $\tau \leq t+\delta,\ a.s.$ and $P_{\tau-}\geq0.$ In
the first equation of (\ref{equ6.7}), the jumps of
$P_t$ are only produced by the random measure $\mu$, from $(
\mathbf{H2.1})$-(ii) and $l\geq0$ on $E$, $$\Delta P_\tau \geq 0, \
P_\tau=P_{\tau-}+\Delta P_\tau \geq0.$$ Therefore, $P_\tau=0,$ when $\tau<
t+\delta,$ and $P_\tau\geq0$, when $\tau=t+\delta.$ Consider the
following FBSDE on $[\tau, t + \delta]$:
\begin{equation}\label{equ6.8}\left \{ \begin{array}{llll}
d \tilde{P}_s &=& (g^{2}_s \tilde{P}_s - b^2_s \tilde{Q}_s
-\sigma^2_s \tilde{M}_s-h^2_s\tilde{N}_s)ds + (g^{3}_s \tilde{P}_s-
b^3_s \tilde{Q}_s -\sigma^3_s
\tilde{M}_s-h^3_s\tilde{N}_s)dB_s+\int_Eg_s^4\tilde{P}_{s-}l(e)\widetilde{\mu}(dsde), \\
d\tilde{Q}_s&=& (g^{1}_s \tilde{P}_s - b^1_s \tilde{Q}_s -
\sigma^1_s\tilde{M}_s-h^1_s\tilde{N}_s)ds - \tilde{M}_sdB_s-\int_E\tilde{N}_{s-}\widetilde{\mu}(dsde), \\
\tilde{P}_\tau  &=& 0, \\
\tilde{M}_{t+\delta} &=&
-\bar{\Phi}\tilde{P}_{t+\delta}.
\end{array}\right.\end{equation}
Due to Theorem \ref{pro6.3} there exists $0 < \delta_3 \leq
\delta_2$ such that for every $ 0 \leq \delta \leq \delta_3$,
(\ref{equ6.8}) has a unique solution
$(\tilde{P},\tilde{Q},\tilde{M}_s,\tilde{N}_s)$ on $[\tau, t +
\delta]$. Clearly, $( \tilde{P}_s,
\tilde{Q}_s,\tilde{M}_s,\tilde{N}_s)\equiv (0, 0, 0,0)$ is the
unique solution of (\ref{equ6.8}). Let
$$\begin{array}{llll}&&\bar{P}_s =
I_{[t,\tau ]}(s)P_s + I_{(\tau,t+\delta]}(s) \tilde{P}_s,\ \ \ \ \bar{Q}_s= I_{[t,\tau ]}(s)Q_s+ I_{(\tau,t+\delta]}(s)\tilde{Q}_s, \\
&&\bar{M}_s = I_{[t,\tau ]}(s)M_s + I_{(\tau,t+\delta]}(s)\tilde{M}_s,\ \bar{N}_s = I_{[t,\tau ]}(s)N_s +
I_{(\tau,t+\delta]}(s)\tilde{N}_s,\  \ s \in[t, t +
\delta].\end{array}$$ Considering that $P_\tau=0$ on
$\{\tau<t+\delta\}$, it's easy to show that $(\bar{P},
\bar{Q},\bar{M},\bar{N})$ is a solution of FBSDE (\ref{equ6.7}).
Therefore, from the uniqueness of solution of FBSDE (\ref{equ6.7})
on $[t, t + \delta]$, where $0\leq \delta \leq \delta_3,$ we have
$\bar{P}_t = P_t = 1 > 0$. Furthermore, from the definition of
$\tau$ we have $\bar{P}_{t+\delta} \geq 0$, \mbox{P-a.s.}, that is,
$P_{t+\delta}\geq 0,$ \mbox{P-a.s.} Therefore, we have $Y^1_t\geq
Y^2_t,\ \mbox{P-a.s.}$

\end{proof}

In order to derive some regularity results, we need the following condition:
\begin{description}
\item[$( \mathbf{H3.4})$]  For any
$t\in[0,T]$,
for any $(x,y,z)\in \mathbb{R}^{n}\times \mathbb{R}\times \mathbb{R}^{d},\  \mbox{P-a.s.},$ $|h(t,x,y,z,e)|\leq \rho(e)(1+|x|+|y|),$ where $\rho(e)=C(1\wedge|e|).$
\end{description}

\begin{theorem}
\label{pro6.4} Let $\Phi$ be deterministic, and suppose the
assumptions $( \mathbf{H2.1}),\ ( \mathbf{H3.1}),\ ( \mathbf{H3.3}),\ ( \mathbf{H3.4})$
hold true. Then, for every $p\geq 2,$ there exists a sufficiently
small constant $\tilde{\delta} > 0$, only depending on the Lipschitz
constants $K$ and $L_\sigma,\ L_h(\cdot),$ and some constant $\tilde
{C}_{p,K}$, only depending on $p$, the Lipschitz constants $K,\
L_\sigma,\ L_h(\cdot)$ and the linear growth constant $L$, such that
for every $0\leq \delta \leq \tilde{\delta}$ and $\zeta \in
L^{p}(\Omega,\mathcal{F}_{t}, P;\mathbb{R}^{n}),$
\[%
\begin{array}
[c]{llll}%
{\rm(i)} &  & E [ \mathop{\rm sup}\limits_{t\leq s\leq t+\delta}
|X^{t,\zeta}_{s} |^{p} + \mathop{\rm sup}\limits_{t\leq s\leq
t+\delta} |Y^{t,\zeta}_{s} |^{p} + (\int_{t}^{t+\delta}
|Z^{t,\zeta}_{s}|^{2}ds)^{\frac{p}{2}} \\
&&+( \int_{t}^{t+\delta}\int_{E} |K^{t,\zeta
}_{s}(e)|^{2}\lambda(de)ds)^{\frac{p}{2}}|\mathcal{F}_{t}]\leq \tilde{C}%
_{p,K}(1 + |\zeta|^{p}),\ \mbox{P-a.s.}; & \\
{\rm(ii)} &  & E [ \mathop{\rm sup}\limits_{t\leq s\leq t+\delta}
|X^{t,\zeta}_{s} -\zeta|^{p} | \mathcal{F}_{t} ] \leq
\tilde{C}_{p,K}\delta(1 +
|\zeta|^{p}),\ \mbox{P-a.s.}, & \\
{\rm(iii)} &  & E [ ( \int_{t}^{t+\delta} |Z^{t,\zeta}_{s}%
|^{2}ds)^{\frac{p}{2}} +( \int_{t}^{t+\delta}\int _{E}
|K^{t,\zeta}_{s}(e)|^{2}\lambda(de)ds)^{\frac{p}{2}} |
\mathcal{F}_{t} ] \leq \tilde{C}_{p,K}\delta^{\frac{p}{2}}(1 +
|\zeta|^{p}),\ \mbox{P-a.s.}, &
\end{array}
\]
where $(X^{t,\zeta}_{s} , Y^{t,\zeta}_{s}
,Z^{t,\zeta}_{s},K^{t,\zeta}_{s} )_{s\in[t,t+\delta]}$ is the
solution of FBSDE (\ref{equ6.4}) associated with $(b,\sigma,g,
\zeta, \Phi)$ and with the time horizon $t + \delta.$
\end{theorem}

\begin{proof}Without loss of generality, we restrict ourselves to the proof for $p=2k,\ k\in \mathbb{Z}^+$.

Due to Theorem \ref{pro6.3}, there exists a constant $\delta_0
> 0$ depending on $K,\
L_\sigma,\ \tilde{C}_h$, such that for every $0\leq \delta \leq
\delta_0,$ (\ref{equ6.4}) has a unique solution on $[t, t +
\delta]$, i.e.,
\begin{equation}\label{ee6.12}Y^{t,\zeta}_s  =
\Phi(X^{t,\zeta}_{t+\delta}) + \int_s^{t+\delta} g(r,X^{t,\zeta}_r ,
Y^{t,\zeta}_r ,Z^{t,\zeta}_r,K^{t,\zeta}_r)dr - \int_s^{t+\delta}
Z^{t,\zeta}_r dB_r-   \int_s^{t+\delta}\int_E K^{t,\zeta}_r(e)
\tilde{\mu}(drde).
\end{equation} Set $\tilde{Y}^{t,\zeta}_s  = Y^{t,\zeta}_s -
\Phi(\zeta ).$ For any $\beta \geq0$,  by applying It\^{o}'s formula
to $e^{\beta s}|\tilde{Y}^{t,\zeta}_s|^2$,  taking $\beta$ large
enough, using BSDE standard methods, and by considering that $|g(r,
\zeta , \Phi(\zeta ), 0,0)| \leq C(1 + |\zeta |)$, we get
\begin{equation}\label{ee6.13}\begin{array}{llll}&& |\tilde{Y}^{t,\zeta}_s |^2 + E[
\int_s^{t+\delta} (|\tilde{Y}^{t,\zeta}_r|^2 + |Z^{t,\zeta}_r |^2+
\int_E|K^{t,\zeta}_r(e)
|^2\lambda(de))dr|\mathcal {F}_s] \\
&\leq& CE[\sup \limits_{s\leq r\leq t+\delta} |X^{t,\zeta}_r -\zeta
|^2|\mathcal {F}_s] + C(t + \delta- s)(1 + |\zeta |^2),\
\mbox{P-a.s.},\end{array}\end{equation} where $C$ only depends on
$K$ and $L$. Therefore, from (\ref{ee6.12}) and (\ref{ee6.13}) and
Burkholder-Davis-Gundy inequality, \begin{equation}\label{ee6.14}E[
\sup \limits_{t\leq s\leq t+\delta} |\tilde{Y}^{t,\zeta}_s
|^2|\mathcal {F}_t] \leq CE[ \sup \limits_{t\leq r\leq t+\delta}
|X^{t,\zeta}_r - \zeta |^2|\mathcal {F}_t] + C\delta(1 + |\zeta
|^2),\  \mbox{P-a.s.}
\end{equation} On the other hand, from (\ref{ee6.13}) \begin{equation}\label{ee6.15}|\tilde{Y}^{t,\zeta}_s
|^2\leq CE[ \sup \limits_{t\leq r\leq t+\delta}|X^{t,\zeta}_r -
\zeta |^2|\mathcal {F}_s] + C\delta(1 + |\zeta |^2),\
\mbox{P-a.s.},\ t \leq s \leq t+\delta.
\end{equation}
When $p
> 2$, we define $\eta = \sup \limits_{t\leq r\leq t+\delta} |X^{t,\zeta}_r- \zeta |^2 \in L^2(\Omega,
\mathcal {F}_{t + \delta},P;\mathbb{R}^n).$ Then $M_s := E[\eta|\mathcal {F}_s], s
\in [t, t + \delta],$ is a martingale, and from Doob's martingale
inequality we have
\begin{equation}\label{ee6.16}E[\sup \limits_{t\leq s\leq t+\delta}
|M_s|^ {\frac{p}{2}} |\mathcal {F}_t]\leq C_pE[|M_{t+\delta}|^
{\frac{p}{2}} |\mathcal {F}_t]\leq C_pE[\eta^ {\frac{p}{2}}
|\mathcal {F}_t] = C_pE[\sup \limits_{t\leq r\leq t+\delta}
|X^{t,\zeta}_r-\zeta |^p|\mathcal {F}_t],\  \mbox{P-a.s.}
\end{equation}
Therefore, from (\ref{ee6.15}) and (\ref{ee6.16})
\begin{equation}\label{ee6.17}E[ \sup \limits_{t\leq s\leq t+\delta}
|\tilde{Y}^{t,\zeta}_s |^p|\mathcal {F}_t] \leq C_pE[ \sup
\limits_{t\leq r\leq t+\delta} |X^{t,\zeta}_r -\zeta |^p|\mathcal
{F}_t] + C_p\delta^{\frac{p}{2}} (1 + |\zeta |^p),\ \mbox{P-a.s.}
\end{equation} Now we consider $$Y^{t,\zeta}_s - \Phi(\zeta ) =
\Phi(X^{t,\zeta}_{t+\delta})- \Phi(\zeta ) +
\int_s^{t+\delta}g(r,X^{t,\zeta}_r , Y^{t,\zeta}_r
,Z^{t,\zeta}_r,K^{t,\zeta}_r )dr- \int_s^{t+\delta}Z^{t,\zeta}_r
dB_r-  \int_s^{t+\delta}\int_EK^{t,\zeta}_r(e) \tilde{\mu}(dsde).$$
From Burkholder-Davis-Gundy inequality and (\ref{ee6.1111}), (\ref{ee6.17}),
\begin{equation}\label{BDG}\begin{array}{llll}&&
E [ ( \int_t^{t+\delta} |Z^{t,\zeta}_s |^2ds)^{\frac{p}{2}} |
\mathcal {F}_t ]+E [ ( \int_t^{t+\delta}\int_E |K^{t,\zeta}_s(e)
|^2\lambda(de)ds)^{\frac{p}{2}} | \mathcal {F}_t ]\\
&\leq&E [ ( \int_t^{t+\delta} |Z^{t,\zeta}_s |^2ds)^{\frac{p}{2}} |
\mathcal {F}_t ]+C_pE [ ( \int_t^{t+\delta}\int_E |K^{t,\zeta}_s(e)
|^2\mu(dsde))^{\frac{p}{2}} | \mathcal {F}_t ]\\
&\leq& C_pE [ \sup \limits_{t\leq s\leq t+\delta} | \int_t^s
Z^{t,\zeta}_r dB_r+\int_t^s\int_E K^{t,\zeta}_r(e)
\tilde{\mu}(dsde)|^p |
\mathcal {F}_t]\\
&\leq& C_pE  [ \sup \limits_{t\leq s\leq t+\delta}
|\tilde{Y}^{t,\zeta}_s |^p + ( \int_t^{t+\delta} |g(s,X^{t,\zeta}_s
, Y^{t,\zeta}_s ,Z^{t,\zeta}_s ,K^{t,\zeta}_s )|ds)^p | \mathcal
{F}_t ]\\
&=& C_pE  [ \sup \limits_{t\leq s\leq t+\delta}
|\tilde{Y}^{t,\zeta}_s  |^p | \mathcal {F}_t ] + C_pE [ ( \int_t^{t+\delta} |g(s,X^{t,\zeta}_s , Y^{t,\zeta}_s
,Z^{t,\zeta}_s,K^{t,\zeta}_s)- g(s, \zeta , \Phi(\zeta ), 0,0) \\
&&+
g(s, \zeta ,
\Phi(\zeta ), 0,0)|ds)^p | \mathcal {F}_t ] \\
&\leq& ( C_p + C_p\delta^p ) E [ \sup \limits_{t\leq s\leq t+\delta}
|X^{t,\zeta}_s -\zeta |^p | \mathcal {F}_t] + C_p\delta^{\frac{p}{2}}
(1 + |\zeta |^p) + C_p\delta^{\frac{p}{2}} E [ ( \int_t^{t+\delta}
|Z^{t,\zeta}_s |^2ds)^{\frac{p}{2}} |
\mathcal {F}_t ]\\
&&+C_p\delta^{\frac{p}{2}}E [ ( \int_t^{t+\delta} \int_E
|K^{t,\zeta}_s(e) |^2\lambda(de)ds)^{\frac{p}{2}} | \mathcal {F}_t ].
\end{array}\end{equation}
By choosing $0 < \delta_1\leq \delta_0$ such that $1 -
C_p\delta^{\frac{p}{2}}_1 > 0$, we get, for any $0\leq \delta \leq
\delta_1, \ \mbox{P-a.s.},$
\begin{equation}\label{ee6.18} \begin{array}{llll} &&E [(
\int_t^{t+\delta} |Z^{t,\zeta}_s |^2ds)^{\frac{p}{2}} | \mathcal
{F}_t ]+E [ ( \int_t^{t+\delta}\int_E |K^{t,\zeta}_s(e)
|^2\lambda(de)ds)^{\frac{p}{2}} | \mathcal {F}_t ]\\
&&\leq ( C_p + C_p\delta^p ) E [ \sup \limits_{t\leq s\leq T}|X^{t,\zeta}_s - \zeta |^p | \mathcal {F}_t ] +C_p\delta^{\frac{p}{2}} (1 + |\zeta |^p).\end{array}\end{equation}
Therefore, from the second line and the latter estimate of (\ref{BDG}) we know
\begin{equation}\label{ee6.28}\begin{array}{llll} &&
E [ ( \int_t^{t+\delta} |Z^{t,\zeta}_s |^2ds)^{\frac{p}{2}} |
\mathcal {F}_t ]+E [ ( \int_t^{t+\delta}\int_E |K^{t,\zeta}_s(e)
|^2\mu(dsde)^{\frac{p}{2}} | \mathcal {F}_t ]\\
&&\leq ( C_p + C_p\delta^p ) E [ \sup \limits_{t\leq s\leq T}|X^{t,\zeta}_s - \zeta |^p | \mathcal {F}_t ] +C_p\delta^{\frac{p}{2}} (1 + |\zeta |^p),\ \mbox{P-a.s.}\end{array}\end{equation}
Similarly, equation (\ref{equ6.4}) and the estimates (\ref{ee6.17}),
(\ref{ee6.18}), (\ref{ee6.28})  yield \begin{equation}\label{ee6.29}\begin{array}{llll}&&E [ \sup \limits_{t\leq
r\leq
t+\delta} |X^{t,\zeta}_r - \zeta |^p | \mathcal {F}_t ] \\
&\leq& C_pE [ ( \int_t^{t+\delta} b(r,X^{t,\zeta}_r , Y^{t,\zeta}_r
,Z^{t,\zeta}_r)dr)^p | \mathcal
{F}_t ] + C_pE [ ( \int_t^{t+\delta} |\sigma(r,X^{t,\zeta}_r ,
Y^{t,\zeta}_r ,Z^{t,\zeta}_r)|^2dr)^{\frac{p}{2}}
| \mathcal {F}_t ]
\\&&+ C_pE [ (
\int_t^{t+\delta}\int_E |h(r,X^{t,\zeta}_{r-} , Y^{t,\zeta}_{r-}
,Z^{t,\zeta}_r,e)|^2\mu(drde)^{\frac{p}{2}} |
\mathcal {F}_t
]\\
&\leq& C_pE [ (\int_t^{t+\delta}(1+|X_r^{t,\zeta}-\zeta|+|\zeta|+|Y_r^{t,\zeta}-\Phi(\zeta)|+|Z_r^{t,\zeta}|)dr)^p|
\mathcal {F}_t]\\
&&+C_pE [ (\int_t^{t+\delta}(1+|X_r^{t,\zeta}-\zeta|^2+|\zeta|^2+|Y_r^{t,\zeta}-\Phi(\zeta)|^2+L_\sigma^2|Z_r^{t,\zeta}|^2)dr)^\frac{p}{2}|
\mathcal {F}_t]\\
&&+C_pE [ (
\int_t^{t+\delta}\int_E |h(r,X^{t,\zeta}_{r-} , Y^{t,\zeta}_{r-}
,Z^{t,\zeta}_r,e)|^2\mu(drde))^{\frac{p}{2}} |
\mathcal {F}_t
],
\end{array}\end{equation}
where
\begin{equation}\label{ee6.300}\begin{array}{llll}&&
E [(\int_t^{t+\delta}\int_E |h(r,X^{t,\zeta}_{r-} , Y^{t,\zeta}_{r-}
,Z^{t,\zeta}_r,e)|^2\mu(drde))^{\frac{p}{2}} |
\mathcal {F}_t]\\
&\leq&  E [(\int_t^{t+\delta}\int_EC(1\wedge |e|^2)(1+|X^{t,\zeta}_{r-}|+|Y^{t,\zeta}_{r-}|)^2\mu(drde))^{\frac{p}{2}} |
\mathcal {F}_t]\\
&\leq&  E [(\int_t^{t+\delta}\int_EC_p(1\wedge |e|^2)(1+|X^{t,\zeta}_{r-}|^2+|Y^{t,\zeta}_{r-}|^2)\mu(drde))^{\frac{p}{2}} |
\mathcal {F}_t]\\
&\leq& E [(\int_t^{t+\delta}\int_EC_p(1\wedge |e|^2)(1+|X^{t,\zeta}_{r-}-\zeta|^2+|Y^{t,\zeta}_{r-}-\Phi(\zeta)|^2+|\zeta|^2)\mu(drde))^{\frac{p}{2}} |
\mathcal {F}_t]\\
&\leq&  E [(\int_t^{t+\delta}\int_EC_p(1\wedge |e|^2)\mu(drde))^{\frac{p}{2}} |
\mathcal {F}_t](1+|\zeta|^p)+C_pE[(\int_t^{t+\delta}\int_E(1\wedge |e|^2)|X^{t,\zeta}_{r-}-\zeta|^2\mu(drde))^{\frac{p}{2}} |\mathcal {F}_t]\\
&&
+C_pE[(\int_t^{t+\delta}\int_E(1\wedge |e|^2)|\tilde{Y}^{t,\zeta}_{r-}|^2\mu(drde))^{\frac{p}{2}} |\mathcal {F}_t]
\\
&\leq&C_p\delta(1+|\zeta|^p)+C_p\delta E[\sup\limits_{t\leq s\leq t+\delta}|X^{t,\zeta}_{r-}-\zeta|^p|\mathcal {F}_t]+C_p\delta E[\sup\limits_{t\leq s\leq t+\delta}|\tilde{Y}^{t,\zeta}_{s}|^p|\mathcal {F}_t],
\end{array}\end{equation}
where we have used that \begin{equation}\label{ee111}E[(\int_t^{t+\delta}\int_E(1\wedge |e|^2)|\tilde{Y}^{t,\zeta}_{r-}|^2\mu(drde))^{\frac{p}{2}} |\mathcal {F}_t]\leq C_p\delta E[\sup\limits_{t\leq s\leq t+\delta}|\tilde{Y}^{t,\zeta}_{s}|^p|\mathcal {F}_t].\end{equation}
Indeed, we denote $\gamma_s(e):=(1\wedge |e|^2)|\tilde{Y}^{t,\zeta}_{s}|^2,\ \bar{\gamma}_s(e):=|\tilde{Y}^{t,\zeta}_{s}|^2,\ A_r:=\int_t^{t+\delta}\int_E \gamma_s(e)\mu(dsde)$. Similarly to (\ref{ee6.11111}) in the proof of Proposition \ref{pro6.2} we can prove that (\ref{ee111}).

In the same way, we  have
$$E[|\int_t^r\int_E|X_{s-}^{t,\zeta}-\zeta|^{2}(1\wedge |e|^2)\mu(dsde)|^\frac{p}{2}|\mathcal {F}_t]\leq C_p(r-t)E[\sup\limits_{t\leq s\leq r}|X_s^{t,\zeta}-\zeta|^{p}|\mathcal {F}_t],$$
and $$E[|\int_t^r\int_E(1\wedge |e|^2)\mu(dsde)|^\frac{p}{2}|\mathcal {F}_t]\leq C_p(r-t).$$
From (\ref{ee6.29}), we have \begin{equation}\label{ee6.31}
\begin{array}{llll}
E [ \sup \limits_{t\leq r\leq t+\delta} |X^{t,\zeta}_r - \zeta |^p | \mathcal {F}_t ]&\leq&C_p(1+|\zeta|^p)\delta+C_p(\delta^\frac{p}{2}+L_\sigma^p)E[(\int_t^{t+\delta}|Z^{t,\zeta}_r|^2dr)^\frac{p}{2}| \mathcal {F}_t]\\
&&+C_p\delta E [ \sup \limits_{t\leq r\leq t+\delta} |X^{t,\zeta}_r - \zeta |^p | \mathcal {F}_t ]+C_p\delta E [ \sup \limits_{t\leq r\leq t+\delta} |Y^{t,\zeta}_r -\Phi( \zeta) |^p | \mathcal {F}_t].
\end{array}
\end{equation}
From (\ref{ee6.17}), (\ref{ee6.28}) and (\ref{ee6.31}), we get
\begin{equation}\begin{array}{llll}
&&E [\sup\limits_{t\leq r\leq t+\delta} |X^{t,\zeta}_r - \zeta |^p | \mathcal {F}_t]\\
&\leq &C_p\delta(1+|\zeta|^p)+C_p(\delta+\delta^\frac{p}{2}+\delta^\frac{3p}{2}+L_{\sigma}^p+L_{\sigma}^p\delta^p)E [\sup\limits_{t\leq r\leq t+\delta} |X^{t,\zeta}_r - \zeta |^p | \mathcal {F}_t],\ \mbox{P-a.s.},\ 0\leq\delta\leq \delta_0.
\end{array}\end{equation}
Due to  $L_{\sigma}$ is sufficiently small, we choose $L_{\sigma}$ satisfying $C_pL_{\sigma}^p<1$. Then there exists a constant $0<\delta_2\leq\delta_1$ such that $1-C_p(\delta_2+\delta_2^\frac{p}{2}+\delta_2^\frac{3p}{2}+L_{\sigma}^p+L_{\sigma}^p\delta_2^p)>0$, therefore we get for any $0\leq\delta\leq \delta_2,$ P-a.s.
\begin{equation}\label{ee6.32}
E [\sup\limits_{t\leq r\leq t+\delta} |X^{t,\zeta}_r - \zeta |^p | \mathcal {F}_t]\leq C_p\delta(1+|\zeta|^p).
\end{equation}\\
Furthermore, from (\ref{ee6.17}), (\ref{ee6.18}), (\ref{ee6.28}), and (\ref{ee6.32}),
\begin{equation}\label{ee6.33}
E [\sup\limits_{t\leq s\leq t+\delta} |X^{t,\zeta}_s|^p+\sup\limits_{t\leq s\leq t+\delta} |Y^{t,\zeta}_s|^p+(\int_t^{t+\delta}|Z^{t,\zeta}_s|^2ds)^\frac{p}{2}+(\int_t^{t+\delta}\int_E |K^{t,\zeta}_s(e)|^2\lambda(de)ds)^\frac{p}{2} | \mathcal {F}_t]\leq \tilde{C}_{p,K}(1+|\zeta|^p),\ \mbox{P-a.s.},
\end{equation}
and $$E[(\int_t^{t+\delta}\int_E |K^{t,\zeta}_s(e)|^2\mu(dsde))^\frac{p}{2} | \mathcal {F}_t]\leq \tilde{C}_{p,K}(1+|\zeta|^p),\ \mbox{P-a.s.}$$
 Now we prove (iii). For convenience, we denote
$$\begin{array}{llll}&&\theta_s^{t,\zeta}:=\int_t^s\int_E h(r,X^{t,\zeta}_{r-} , Y^{t,\zeta}_{r-},Z^{t,\zeta}_r,e)\mu(drde),\ \eta_s^{t,\zeta}(e):=\Phi(\zeta+\theta_{s-}^{t,\zeta}+h(s,\Pi_{s-}^{t,\zeta},e))-\Phi(\zeta+\theta_{s-}^{t,\zeta}),\\
&& \widehat{Y}_s^{t,\zeta}:=\int_t^s\int_E\eta_r^{t,\zeta}(e)\tilde{\mu}(drde),\ \widetilde{X}_s^{t,\zeta}:=X_s^{t,\zeta}-\theta_s^{t,\zeta},\  \widetilde{Y}_s^{t,\zeta}:=Y_s^{t,\zeta}-\widehat{Y}_s^{t,\zeta}-\Phi(\zeta),\\
&& \widetilde{K}_s^{t,\zeta}(e):=K_s^{t,\zeta}(e)-\eta_s^{t,\zeta}(e),\ \Pi^{t,\zeta}_{s-}:=(X^{t,\zeta}_{s-},Y^{t,\zeta}_{s-},Z^{t,\zeta}_{s}).
\end{array}$$
We know \begin{equation}\label{ee6.3888}|\eta_s^{t,\zeta}(e)|\leq C|h(s,\Pi_{s-}^{t,\zeta},e)|\leq C(1\wedge |e|^2)(1+|X_{s-}^{t,\zeta}|+|Y_{s-}^{t,\zeta}|).\end{equation}
And it is easy to check
$$
\begin{array}{llll}
&&\Phi(\zeta+\theta_{t+\delta}^{t,\zeta})-\Phi(\zeta)=\sum\limits_{t<s\leq t+\delta}(\Phi(\zeta+\theta_s^{t,\zeta})-\Phi(\zeta+\theta_{s-}^{t,\zeta}))\\
&=&\sum\limits_{t<s\leq t+\delta}\int_E [\Phi(\zeta+\theta_{s-}^{t,\zeta}+h(s,\Pi_{s-}^{t,\zeta},e))-\Phi(\zeta+\theta_{s-}^{t,\zeta})]\mu(\{s\},de)\\
&=&\int_t^{t+\delta}\int_E[\Phi(\zeta+\theta_{s-}^{t,\zeta}+h(s,\Pi_{s-}^{t,\zeta},e))-\Phi(\zeta+\theta_{s-}^{t,\zeta})]\mu(dsde)=\int_t^{t+\delta}\int_E\eta_s^{t,\zeta}(e)\mu(dsde).
\end{array}$$
Therefore, $$\Phi(\zeta+\theta_{t+\delta}^{t,\zeta})-\Phi(\zeta)-\int_t^{t+\delta}\int_E\eta_s^{t,\zeta}(e)\lambda(de)ds=\int_t^{t+\delta}\int_E\eta_s^{t,\zeta}(e)\tilde{\mu}(dsde).$$
Then, $$\widehat{Y}_s^{t,\zeta}=\Phi(\zeta+\theta_{t+\delta}^{t,\zeta})-\Phi(\zeta)-\int_t^{t+\delta}\int_E\eta_s^{t,\zeta}(e)\lambda(de)ds-\int_s^{t+\delta}\int_E\eta_s^{t,\zeta}(e)\tilde{\mu}(dsde).$$
From equation (\ref{equ6.4}), we have
\begin{equation}
\label{equ6.444}\left \{
\begin{array}
[c]{llll}%
d\widetilde{X}_s^{t,\zeta} & = & b(s,\widetilde{X}_s^{t,\zeta}+\theta_s^{t,\zeta}, \widetilde{Y}_s^{t,\zeta}+\widehat{Y}_s^{t,\zeta}+\Phi(\zeta), Z_s^{t,\zeta})ds + \sigma(s,\widetilde{X}_s^{t,\zeta}+\theta_s^{t,\zeta}, \widetilde{Y}_s^{t,\zeta}+\widehat{Y}_s^{t,\zeta}+\Phi(\zeta), Z_s^{t,\zeta})dB_{s}\\
&&+\int_{E}h(s,\widetilde{X}_{s-}^{t,\zeta}+\theta_s^{t,\zeta}, \widetilde{Y}_{s-}^{t,\zeta}+\widehat{Y}_{s-}^{t,\zeta}+\Phi(\zeta), Z_s^{t,\zeta},e)\lambda(de)ds,  \ \ \  s\in[t, t+\delta], & \\
d\widetilde{Y}_s^{t,\zeta} & = & -g(s,\widetilde{X}_s^{t,\zeta}+\theta_s^{t,\zeta}, \widetilde{Y}_s^{t,\zeta}+\widehat{Y}_s^{t,\zeta}+\Phi(\zeta),Z_s^{t,\zeta},\widetilde{K}_s^{t,\zeta}+\eta_s^{t,\zeta})ds + Z_s^{t,\zeta}dB_{s}+\int
_{E}\widetilde{K}_s^{t,\zeta}(e)\tilde{\mu}(dsde), \\
\widetilde{X}_t^{t,\zeta} & = & \zeta, \\
\widetilde{Y}_{ t+\delta}^{t,\zeta} & = & \Phi(\widetilde{X}_{t+\delta}^{t,\zeta}+\theta_{t+\delta}^{t,\zeta})-\Phi(\zeta+\theta_{t+\delta}^{t,\zeta})+\int_t^{t+\delta}\int_E\eta_s^{t,\zeta}(e)\lambda(de)ds. &
\end{array}
\right.
\end{equation}
For $(X^{t,\zeta},Y^{t,\zeta},Z^{t,\zeta},K^{t,\zeta})$, (\ref{ee6.33}) holds true, for any $\delta\in[0,\tilde{\delta}].$
For the backward part of equation (\ref{equ6.444}),
$$|\widetilde{Y}_s^{t,\zeta}|\leq E[|\widetilde{Y}_{t+\delta}^{t,\zeta}|+\int_t^{t+\delta}|g(r,X_r^{t,\zeta}, Y_r^{t,\zeta},Z_r^{t,\zeta},K_r^{t,\zeta})|dr\mid\mathcal {F}_s],\ \ s\in[t,t+\delta].$$
From (\ref{ee6.33}) and Doob's martingale inequality for $p>2$,
\begin{equation}\label{ee6.400}\begin{array}{llll}E[\sup\limits_{t\leq s\leq t+\delta}|\widetilde{Y}_s^{t,\zeta}|^p\mid\mathcal {F}_t]&\leq& C_pE[|\widetilde{Y}_{t+\delta}^{t,\zeta}|^p\mid\mathcal {F}_t]+C_p\delta^\frac{p}{2}E[(\int_t^{t+\delta}|g(r,X_r^{t,\zeta}, Y_r^{t,\zeta},Z_r^{t,\zeta},K_r^{t,\zeta})|^2dr)^\frac{p}{2}\mid\mathcal {F}_t]\\
&\leq& C_p\delta^\frac{p}{2}(1+|\zeta|^p)+C_pE[|\widetilde{Y}_{t+\delta}^{t,\zeta}|^p\mid\mathcal {F}_t].
\end{array}\end{equation}
We need to estimate $|\widetilde{Y}_{t+\delta}^{t,\zeta}|$, and notice \begin{equation}\label{ee3.33}|\widetilde{Y}_{t+\delta}^{t,\zeta}|\leq C|\widetilde{X}_{t+\delta}^{t,\zeta}-\zeta|+\int_t^{t+\delta}\int_E|\eta_s^{t,\zeta}(e)|\lambda(de)ds.\end{equation}
From (\ref{ee6.33}) and (\ref{ee6.3888}), we get
\begin{equation} \label{6.422}E[(\int_t^{t+\delta}\int_E|\eta_s^{t,\zeta}(e)|\lambda(de)ds)^p\mid\mathcal {F}_t]\leq C_p\delta^\frac{p}{2}(1+|\zeta|^p).
\end{equation}
On the other hand, from (\ref{ee6.33}), we have
\begin{equation}\label{ee6.4404}\begin{array}{llll}
&&E[|\int_t^{s}|b(s,X_{r}^{t,\zeta},Y_{r}^{t,\zeta},Z_{r}^{t,\zeta})|dr|^p\mid\mathcal {F}_t]
\leq  C_p\delta^\frac{p}{2}(1+|\zeta|^p),\\
&&E[|\int_t^{s}|\sigma(s,X_{r}^{t,\zeta},Y_{r}^{t,\zeta},Z_{r}^{t,\zeta})|dB_r|^p\mid\mathcal {F}_t]
\leq  C_p\delta^\frac{p}{2}(1+|\zeta|^p)+C_pL_\sigma^pE[(\int_t^{t+\delta}|Z_{r}^{t,\zeta}|^2dr)^\frac{p}{2}\mid\mathcal {F}_t],\\
&&E[(\int_t^{t+\delta}\int_E|h(s,X_{s-}^{t,\zeta},Y_{s-}^{t,\zeta},Z_{s}^{t,\zeta},e)|\lambda(de)ds)^p\mid\mathcal {F}_t]\\
&&\leq  C_pE[(\int_E(1\wedge|e|^2)\lambda(de))^p(\int_t^{t+\delta}(1+|X_{s}^{t,\zeta}|+|Y_{s}^{t,\zeta}|)ds)^p\mid\mathcal {F}_t]\leq C_p\delta^\frac{p}{2}(1+|\zeta|^p).
\end{array}
\end{equation}
From (\ref{equ6.444}) and the above estimates (\ref{ee6.4404}),
we know
\begin{equation}\label{ee6.445}\begin{array}{llll}
E[\sup\limits_{t\leq s\leq t+\delta}|\widetilde{X}_{t+\delta}^{t,\zeta}-\zeta|^p\mid\mathcal {F}_t]\leq  C_p\delta^\frac{p}{2}(1+|\zeta|^p)+C_pL_\sigma^pE[(\int_t^{t+\delta}|Z_{r}^{t,\zeta}|^2dr)^\frac{p}{2}\mid\mathcal {F}_t].
\end{array}
\end{equation}
From (\ref{ee6.400}), (\ref{ee3.33}), (\ref{6.422}) and (\ref{ee6.445}),
 \begin{equation}\label{ee6.488}
 E[\sup\limits_{t\leq s\leq t+\delta}|\widetilde{Y}_s^{t,\zeta}|^p\mid\mathcal {F}_t]
 \leq  C_p\delta^\frac{p}{2}(1+|\zeta|^p)+C_pL_\sigma^pE[(\int_t^{t+\delta}|Z_{r}^{t,\zeta}|^2dr)^\frac{p}{2}\mid\mathcal {F}_t].
 \end{equation}
 From Burkholder-Davis-Gundy inequality and (\ref{ee6.33}), we have
  \begin{equation}\label{ee6.499}\begin{array}{llll}
  &&E[(\int_t^{t+\delta}|Z_{r}^{t,\zeta}|^2dr)^\frac{p}{2}\mid\mathcal {F}_t]+E[(\int_t^{t+\delta}\int_E|\widetilde{K}_{r}^{t,\zeta}(e)|^2\mu(drde))^\frac{p}{2}\mid\mathcal {F}_t]\\
  &\leq&C_pE[\sup\limits_{t\leq s\leq t+\delta}|\int_t^sZ_{r}^{t,\zeta}dB_r+\int_t^s\int_E\widetilde{K}_{r}^{t,\zeta}(e)\tilde{\mu}(drde)|^p\mid\mathcal {F}_t]\\
  &\leq&C_pE[\sup\limits_{t\leq s\leq t+\delta}|\widetilde{Y}_s^{t,\zeta}|^p\mid\mathcal {F}_t]+C_pE[(\int_t^{t+\delta}|g(r,X_{r}^{t,\zeta},Y_{r}^{t,\zeta},Z_{r}^{t,\zeta},K_{r}^{t,\zeta})|dr)^p\mid\mathcal {F}_t]\\
  &\leq&C_pE[\sup\limits_{t\leq s\leq t+\delta}|\widetilde{Y}_s^{t,\zeta}|^p\mid\mathcal {F}_t]+C_p\delta^\frac{p}{2}E[(\int_t^{t+\delta}(1+|X_{r}^{t,\zeta}|^2+|Y_{r}^{t,\zeta}|^2+|Z_{r}^{t,\zeta}|^2)dr)^\frac{p}{2}\mid\mathcal {F}_t]\\
  &&+C_pE[(\int_t^{t+\delta}\int_E|K_{r}^{t,\zeta}(e)|(1\wedge|e|)\lambda(de)dr)^p\mid\mathcal {F}_t]\\
  &\leq& C_p\delta^\frac{p}{2}(1+|\zeta|^p)+C_pL_\sigma^pE[(\int_t^{t+\delta}|Z_{r}^{t,\zeta}|^2dr)^\frac{p}{2}\mid\mathcal {F}_t].
 \end{array}\end{equation}
 As $L_\sigma$ is sufficiently small, for $C_pL_\sigma^p<1,$ we have
 \begin{equation}\label{ee6.599}
  E[(\int_t^{t+\delta}|Z_{r}^{t,\zeta}|^2dr)^\frac{p}{2}\mid\mathcal {F}_t]+E[(\int_t^{t+\delta}\int_E|\widetilde{K}_{r}^{t,\zeta}(e)|^2\mu(drde))^\frac{p}{2}\mid\mathcal {F}_t]
  \leq  C_p\delta^\frac{p}{2}(1+|\zeta|^p).
 \end{equation}
 From (\ref{ee6.1111}), (\ref{6.422}), (\ref{ee6.599}) and $K_s^{t,\zeta}(e)=\widetilde{K}_s^{t,\zeta}(e)+\eta_s^{t,\zeta}(e)$, we get
 $$E[(\int_t^{t+\delta}\int_E|K_{r}^{t,\zeta}(e)|^2\lambda(de)dr)^\frac{p}{2}\mid\mathcal {F}_t]
  \leq  C_p\delta^\frac{p}{2}(1+|\zeta|^p).$$
 Therefore, the estimate (iii) is derived.

\end{proof}

\begin{remark}
If the initial state $\zeta=x\in \mathbb{R}^n$ is given, the terminal condition $\Phi$ becomes $\Phi(x)$, that is, FBSDE (\ref{equ6.4}) becomes the following
\begin{equation}
\label{equ6.12124}\left \{
\begin{array}
[c]{llll}%
dX_{s} & = & b(s,X_{s}, Y_{s}, Z_{s})ds + \sigma(s,X_{s}, Y_{s}, Z_{s}%
)dB_{s}+\int_{E}h(s,X_{s-}, Y_{s-}, Z_{s},e)\tilde{\mu
}(dsde), & \\
dY_{s} & = & -g(s,X_{s}, Y_{s},Z_{s},K_{s})ds + Z_{s}dB_{s}+\int
_{E}K_{s}(e)\tilde{\mu}(dsde),\ s \in[t, t+\delta], & \\
X_{t} & = & x, & \\
Y_{ t+\delta} & = & \Phi(x), &
\end{array}
\right.
\end{equation}
then Theorem \ref{pro6.4} still holds.

Indeed, from Lemma \ref{le2.2}, FBSDE (\ref{equ6.12124}) has a unique solution $(X,Y,Z,K)$. We consider the following FBSDE:
\begin{equation}\left \{
\begin{array}
[c]{llll}%
d\widehat{X}_{s} & = & b(s,\widehat{X}_{s}, \widehat{Y}_{s}+\Phi(x), \widehat{Z}_{s})ds + \sigma(s,\widehat{X}_{s}, \widehat{Y}_{s}+\Phi(x), \widehat{Z}_{s}
)dB_{s}+\int_{E}h(s,\widehat{X}_{s-}, \widehat{Y}_{s-}+\Phi(x), \widehat{Z}_{s},e)\tilde{\mu
}(dsde), & \\
d\widehat{Y}_{s} & = & -g(s,\widehat{X}_{s}, \widehat{Y}_{s}+\Phi(x),\widehat{Z}_{s},\widehat{K}_{s})ds + \widehat{Z}_{s}dB_{s}+\int
_{E}\widehat{K}_{s}(e)\tilde{\mu}(dsde),\ s \in[t, t+\delta], & \\
\widehat{X}_{t} & = & x, & \\
\widehat{Y}_{ t+\delta} & = & 0.&
\end{array}
\right.
\end{equation}
From  Lemma \ref{le2.2}, we know $(X,Y,Z,K)=(\widehat{X},\widehat{Y}+\Phi(x),\widehat{Z},\widehat{K})$. For $(\widehat{X},\widehat{Y},\widehat{Z},\widehat{K})$, Theorem \ref{pro6.4}  holds, which means those estimates in Theorem \ref{pro6.4} still holds for $(X,Y,Z,K)$.

\end{remark}

\begin{proposition}
\label{pro6.5} Suppose that $(b_{i},\sigma_{i},g_{i},\Phi_{i}),\
i=1,2,$ all satisfy the assumptions $( \mathbf{H2.1}),\ (
\mathbf{H3.1}),\ ( \mathbf{H3.3})$. Then from Theorem \ref{th6.2} there exists a constant $0<
\delta_{0}$, only depending on the Lipschitz constants $K,\
L_\sigma$ and $L_h(\cdot)$, such that for  $0\leq \delta \leq
\delta_{0}$, and
the same initial state $\zeta \in L^{2}(\Omega,\mathcal{F}_{t},P;\mathbb{R}%
^{n})$, $(X^{i}_{s}, Y^{i}_{s}, Z^{i}_{s})_{s\in[t, t+\delta]}$ is the
solution of FBSDE (\ref{equ6.1}) associated with $(b_{i},\sigma_{i},g_{i}%
,\Phi_{i})$ on the time interval $[t, t+\delta],\ i=1,2$. It follows that there exists a constant $\delta_{1}>0,$ such that for every
$0\leq \delta \leq \delta_{1},$
\[%
\begin{array}
[c]{llll}
&&|Y^{1}_{t}-Y^{2}_{t}|^{2}\\
& &\leq  CE[|\Phi_{1}(t+\delta,X_{t+\delta}%
^{1})-\Phi_{2}(t+\delta,X_{t+\delta}^{1})|^{2}\mid \mathcal{F}_t] +C\delta E[\int_{t}^{t+\delta}|(b_{1}-b_{2})(s,X_{s}%
^{1},Y_{s}^{1},Z_{s}^{1},K_{s}^{1})|^{2}ds\mid \mathcal{F}_{t}] & \\
&  & +C E[\int_{t}^{t+\delta}|(\sigma_{1}-\sigma_{2}%
)(s,X_{s}^{1},Y_{s}^{1},Z_{s}^{1},K_{s}^{1})|^{2}ds\mid \mathcal{F}_{t}]+C\delta E[\int_{t}^{t+\delta}|(g_{1}-g_{2})(s,X_{s}%
^{1},Y_{s}^{1},Z_{s}^{1},K_{s}^{1})|^{2}ds\mid \mathcal{F}_{t}] & \\
&  & +C E[\int_{t}^{t+\delta}\int_{E}|(h_{1}%
-h_{2})(s,X_{s}^{1},Y_{s}^{1},Z_{s}^{1},K_{s}^{1}(e),e)|^{2}\lambda
(de)ds\mid \mathcal{F}_{t}],\
\mbox{P-a.s.} &
\end{array}
\]

\end{proposition}

For the proof, it is similar to the proof of  Proposition 6.6 in Li, Wei
\cite{LW}.

\begin{remark}
When $(b_{1},\sigma_{1},h_1, f_{1})=(b_{2},\sigma_{2},h_2,f_{2})$ in
Proposition \ref{pro6.5}, we have
\[
|Y^{1}_{t}-Y^{2}_{t}|\leq C(E[|\Phi_{1}(t+\delta,X^{1}_{t+\delta})-\Phi
_{2}(t+\delta,X^{1}_{t+\delta})|^{2}|\mathcal{F}_{t}])^{{\frac{1}{2}}
},\  \mbox{P-a.s.}
\]
\end{remark}

\begin{corollary}
Under the assumptions $( \mathbf{H2.1}),\ ( \mathbf{H3.1}),\ ( \mathbf{H3.3}%
)$, there exists a constant $0< \delta_{0}$, only depending on the
Lipschitz constants $K,\ L_\sigma$ and $L_h(\cdot)$,\ such that for
every $0\leq \delta \leq \delta_{0}$, $\zeta \in
L^{2}(\Omega,\mathcal{F}_{t},P;\mathbb{R}^{n})$ and $\varepsilon>0$,
if $(X^{t, \zeta}_{s}, Y^{t, \zeta}_{s}, Z^{t, \zeta}_{s}, K^{t,
\zeta}_{s})_{s\in[t, t+\delta]}$ is the solution of FBSDE
(\ref{equ6.1}) associated with $(b, \sigma, f, \zeta, \Phi)$, and
$(\overline{X}^{t, \zeta }_{s}, \overline{Y}^{t, \zeta}_{s},
\overline{Z}^{t, \zeta}_{s}, \overline {K}^{t, \zeta}_{s})_{s\in[t,
t+\delta]}$ is that of FBSDE (\ref{equ6.1}) associated with $(b,
\sigma, f, \zeta, \Phi+\varepsilon)$\ on the time interval $[t,
t+\delta]$, then we have that
\[
|Y^{t, \zeta}_{t}-\overline{Y}^{t, \zeta}_{t}|\leq C\varepsilon
,\  \mbox{P-a.s.}
\]
\end{corollary}

 \end{document}